\documentclass[sn-mathphys-num]{sn-jnl}

\usepackage{graphicx}%
\usepackage{multirow}%
\usepackage{amsmath,amssymb,amsfonts}%
\usepackage{amsthm}%
\usepackage{mathrsfs}%
\usepackage[title]{appendix}%
\usepackage{xcolor}%
\usepackage{textcomp}%
\usepackage{manyfoot}%
\usepackage{booktabs}%
\usepackage{algorithm}%
\usepackage{algorithmicx}%
\usepackage{algpseudocode}%
\usepackage{listings}%

\usepackage{tikz}
\usepackage{subcaption}

\newtheorem{theorem}{Theorem}[section]
\newtheorem{lemma}[theorem]{Lemma}

\newtheorem{corollary}[theorem]{Corollary}

\renewcommand{\proofname}{Proof}

\newtheorem{remark}{Remark}%

\newcommand{\dd}{\,\mathrm{d} }

\begin{document}

\title[Eigenvalues of the magnetic Dirichlet Laplacian with constant magnetic field on disks in the strong field limit]{Eigenvalues of the magnetic Dirichlet Laplacian with constant magnetic field on disks in the strong field limit}

\author*[1]{\fnm{Matthias} \sur{Baur} }
\email{matthias.baur@mathematik.uni-stuttgart.de}

\author[1]{\fnm{Timo} \sur{Weidl}}

\affil[1]{\orgdiv{Institute of Analysis, Dynamics and Modeling, Department of Mathematics}, \orgname{University of Stuttgart}, \orgaddress{\street{Pfaffenwaldring 57}, \postcode{70569} \city{Stuttgart},  \state{BW}, \country{Germany}}}

\abstract{We consider the magnetic Dirichlet Laplacian with constant magnetic field on domains of finite measure. First, in the case of a disk, we prove that the eigenvalue branches with respect to the field strength behave asymptotically linear with an exponentially small remainder term as the field strength goes to infinity. We compute the asymptotic expression for this remainder term. Second, we show that for sufficiently large magnetic field strengths, the spectral bound corresponding to the P\'olya conjecture for the non-magnetic Dirichlet Laplacian is violated up to a sharp excess factor which is independent of the domain.}




\maketitle

\section{Introduction}\label{sec1}

In this paper, we consider the magnetic Laplacian with a constant magnetic field on an open set $\Omega \subset \mathbb{R}^2$ of finite Lebesgue measure. Let the vector potential $A:\mathbb{R}^2 \rightarrow \mathbb{R}^2$ be given by the standard linear gauge 
\begin{align}
A(x_1,x_2)= \frac{B}{2}\begin{pmatrix}-x_2 \\ x_1\end{pmatrix} \nonumber
\end{align}
where the scalar $B \in\mathbb{R}$ is the strength of the constant magnetic field. The magnetic Laplacian $H_B^\Omega$ with constant magnetic field and Dirichlet boundary condition is realized by Friedrichs extension of the quadratic form
\begin{align}
a[u]:= \int_\Omega |(-i\nabla +A) u |^2 \dd x, \qquad u \in C_0^\infty(\Omega). \nonumber
\end{align}
Just as the non-magnetic Dirichlet Laplacian, the operator defined via $a[\, \cdot \, ]$ is positive, self-adjoint and has compact resolvent. It admits a sequence of real, positive eigenvalues $\lbrace \lambda_n(\Omega,B) \rbrace_{n\in\mathbb{N}}$ of finite multiplicity that accumulate at infinity only. The corresponding eigenfunctions $\lbrace \psi_n \rbrace_{n\in\mathbb{N}}$ satisfy the weak eigenvalue equation
\begin{align}
\int_\Omega (-i\nabla +A) \psi_n \cdot \overline{(-i\nabla +A) \varphi }  \dd x =\lambda_n(\Omega,B) \int_\Omega \psi_n \overline{\varphi} \dd x, \qquad  \varphi \in C_0^\infty(\Omega). \nonumber
\end{align}
For domains $\Omega$ with regular boundary, in particular for disks, this is equivalent to the strong eigenvalue equation
\begin{align}
(-i\nabla +A)^2 \psi_n  & = \lambda_n(\Omega,B) \psi_n  \qquad \, \text{in } \Omega,  \nonumber \\
\psi_n &= 0 \qquad \qquad  \qquad \text{on } \partial \Omega, \nonumber
\end{align}
where $(-i\nabla +A)^2$ abbreviates the classical differential operator $(-i \partial_1 + A_1)^2 + (-i \partial_2 + A_2)^2$. As usual, the eigenvalues $\lambda_n(\Omega,B)$ shall be sorted by magnitude, i.e.
\begin{align}
0 < \lambda_1(\Omega,B) \leq \lambda_2(\Omega,B)\leq ... \leq \lambda_n(\Omega,B) \leq ... , \nonumber
\end{align}
counting multiplicities. From the eigenvalues of $H_B^\Omega$ one defines for any $\gamma \geq 0$ the so-called Riesz means or eigenvalue moments
\begin{align}
\mathrm{tr} (H_B^\Omega - \lambda)_-^\gamma &:= \begin{cases} \# \lbrace n \, : \, \lambda_n(\Omega,B) < \lambda  \rbrace, &   \quad \gamma =0, \nonumber \\
  \sum\limits_{n: \lambda_n(\Omega,B)<\lambda} (\lambda - \lambda_n(\Omega,B))^\gamma , & \quad \gamma >0.
\end{cases} \nonumber
\end{align}
If $\gamma = 0$, the expression $\mathrm{tr} (H_B^\Omega - \lambda)_-^0$ is more commonly called ''counting function'' and denoted by $N(H_B^\Omega, \lambda)$. 

It is well known that the eigenvalues $\lbrace \lambda_n(\Omega,B) \rbrace_{n\in\mathbb{N}}$ as functions over $B\in\mathbb{R}$ can be identified piecewise with real-analytic eigenvalue branches. This is a classical result of analytic perturbation theory, see for example Kato \cite[Chapter VII \S 3 and \S 4]{Kato1995}. In this framework, the operators $\{ H_B^\Omega \}$ form a type (B) self-adjoint holomorphic family. The eigenvalue branches that represent the spectra of the family $\{ H_B^\Omega \}$ typically do not maintain a particular order, since different branches can intersect. 

We are interested in the behavior of the spectrum of $H_B^\Omega$ as the field strength $B$ becomes large. Our first result (Theorem \ref{thm:eigenvalue_asymptotics_disk}) deals with the special case where $\Omega$ is a disk. Here, all real-analytic eigenvalue branches of the spectra of $\lbrace H_B^\Omega \rbrace_{B\in\mathbb{R}}$ are given in terms of roots of confluent hypergeometric functions. We compute two term asymptotics of all analytic eigenvalue branches. This result generalizes a theorem by Helffer and Persson Sundqvist \cite{Helffer2017}.

In the second part of the paper, we are concerned with spectral bounds on the sorted eigenvalues $\lambda_n(\Omega,B) $ as well as the Riesz means $\mathrm{tr} (H_B^\Omega - \lambda)_-^\gamma$. To locate our work in the existing literature, let us briefly summarize important related results. 

For the non-magnetic Dirichlet Laplacian, i.e.\ the case $B=0$, the celebrated Weyl law \cite{Weyl1912} (see also \cite[Chapter 3.2]{Frank2022}) states that for any domain $\Omega \subset \mathbb{R}^2$
\begin{align}\label{eq:weyl_law}
N(H_0^\Omega, \lambda)= \mathrm{tr} (H_B^\Omega - \lambda)_-^0 = \frac{1}{4\pi}|\Omega| \lambda (1+o(1))  \qquad \text{as } \lambda \rightarrow +\infty.
\end{align}
By integration, one can show that the Riesz means satisfy for any $\gamma\geq 0$
\begin{align} 
\mathrm{tr} (H_0^\Omega - \lambda)_-^\gamma = L_{\gamma,2}^{\mathrm{cl}}  |\Omega| \lambda^{1+\gamma} (1+o(1))  \qquad \text{as } \lambda \rightarrow +\infty  \nonumber
\end{align}
where $L_{\gamma,2}^{\mathrm{cl}} := (4\pi (1+\gamma) )^{-1}$. P\'olya conjectured that the asymptotic expression in \eqref{eq:weyl_law} is actually an upper bound, i.e.\ for any open domain $\Omega \subset \mathbb{R}^2$ holds
\begin{align} 
N(H_0^\Omega, \lambda) \leq  \frac{1}{4\pi}|\Omega| \lambda , \qquad  \lambda \geq 0. \nonumber
\end{align}
or equivalently
\begin{align}  \label{eq:polya_bound}
\lambda_n(\Omega, 0) \geq  \frac{4\pi n}{|\Omega|}, \qquad n\in\mathbb{N}.
\end{align}
He then gave a proof for so-called tiling domains \cite{Polya1961}. However, the case of general domains is still open. Even for disks, P\'olya's conjecture remained unresolved for several decades until it was finally confirmed in 2022 by Filonov et al.\ \cite{Filonov2022a}. 

The Aizenman-Lieb identity \cite{Aizenman1978} allows lifting an inequality for a Riesz mean of order $\gamma \geq 0$ to any order $\gamma' > \gamma$, turning the constant $ L_{\gamma,2}^{\mathrm{cl}}$ into $L_{\gamma',2}^{\mathrm{cl}}$. This means that any domain satisfying P\'olya's inequality also satisfies
\begin{align}
\mathrm{tr} (H_0^\Omega - \lambda)_-^\gamma \leq L_{\gamma,2}^{\mathrm{cl}}  |\Omega| \lambda^{1+\gamma}, \qquad  \lambda \geq 0. \label{eq:riesz_means_bound_DL_conj}
\end{align}
for any $\gamma \geq 0$. As an generalization of P\'olya's conjecture, it is conjectured that \eqref{eq:riesz_means_bound_DL_conj} holds for arbitrary domains.

In 1972 Berezin \cite{Berezin1973} gave a proof of the inequality \eqref{eq:riesz_means_bound_DL_conj} for $\gamma=1$ and arbitrary domains. Together with the Aizenman-Lieb identity, this confirmed the bounds \eqref{eq:riesz_means_bound_DL_conj} for any $\gamma\geq 1$. For $0\leq \gamma <1$ however, the inequalities \eqref{eq:riesz_means_bound_DL_conj} are still unproven. In this case, Berezin's result can still be used to find upper bounds similar to \eqref{eq:riesz_means_bound_DL_conj}, but the derived constants differ from the conjectured constants $L_{\gamma,2}^{\mathrm{cl}}$ by certain excess factors, see for instance \cite[Corollary 3.30]{Frank2022}. Combining the two arguments for $\gamma \geq 1$ and $0\leq \gamma <1$, Berezin's bound implies for $\gamma \geq 0$ the inequalities
\begin{align}
\mathrm{tr} (H_0^\Omega - \lambda)_-^\gamma \leq R_\gamma L_{\gamma,2}^{\mathrm{cl}}  |\Omega| \lambda^{1+\gamma}, \qquad \lambda \geq 0. \label{eq:riesz_means_bound_DL_known}
\end{align}
where 
\begin{align}
R_\gamma = \begin{cases}
2 & \text{ if } \gamma=0, \\
2 \left(\dfrac{\gamma}{\gamma+1}\right)^\gamma  & \text{ if } 0< \gamma < 1,\\
1 & \text{ if } \gamma \geq 1.
\end{cases} \nonumber
\end{align}
The bounds \eqref{eq:riesz_means_bound_DL_known} are believed to be non-sharp for $0\leq \gamma < 1$ since the excess factors $R_\gamma$ can be omitted for all domains satisfying P\'olya's bound, e.g.\ tiling domains and disks. 

Weyl's law \eqref{eq:weyl_law} continues to hold if an additional fixed magnetic vector potential is introduced. This has been proven under minimal assumptions on the magnetic vector potential, see \cite[Theorem A.1]{Frank2009b}. Asymptotics for Riesz means also remain unchanged. It is therefore reasonable to ask whether inequalities similar to \eqref{eq:riesz_means_bound_DL_conj} resp.~\eqref{eq:riesz_means_bound_DL_known} hold with magnetic fields and what the smallest possible values of the excess factors are, if they cannot be omitted. Following up a remark by Helffer, Laptev and Weidl \cite{Laptev2000} proved that \eqref{eq:riesz_means_bound_DL_conj} continues to hold if $H_0^\Omega$ is replaced by a magnetic Laplacian with arbitrary magnetic field in case $\gamma \geq 3/2$. It follows that bounds of the form \eqref{eq:riesz_means_bound_DL_known} also hold if $\gamma < 3/2$, but with larger excess factors than those known in the non-magnetic case, see \cite[Appendix]{Frank2009}, \cite[Theorem 3.1]{Frank2009b}. For the special case of the constant magnetic field, Erd\H{o}s, Loss and Vougalter \cite{Erdoes2000} established \eqref{eq:riesz_means_bound_DL_conj} for $\gamma = 1$. This means that \eqref{eq:riesz_means_bound_DL_known} extends to
\begin{align}
\mathrm{tr} (H_B^\Omega - \lambda)_-^\gamma \leq R_\gamma L_{\gamma,2}^{\mathrm{cl}}  |\Omega| \lambda^{1+\gamma}, \qquad  \lambda \geq 0, B\in\mathbb{R}, \label{eq:riesz_means_bound_ML_known}
\end{align}
with the same excess factors $R_\gamma$ as without magnetic field. 

In contrast to its non-magnetic version \eqref{eq:riesz_means_bound_DL_known}, it has been shown that \eqref{eq:riesz_means_bound_ML_known} is sharp for any $\gamma\geq 0$, see Frank, Loss and Weidl \cite{Frank2009}. Frank \cite[Theorem 3.6]{Frank2009b} even proved that \eqref{eq:riesz_means_bound_ML_known} is sharp for any given domain $\Omega$: The bounds \eqref{eq:riesz_means_bound_ML_known} can be saturated when letting $B$ and $\lambda$ tend to infinity simultaneously in a suitable way. This is remarkable since the origin of the excess factors $R_\gamma$ in \eqref{eq:riesz_means_bound_DL_known} and \eqref{eq:riesz_means_bound_ML_known} is non-spectral. What is believed to be an artifact of a rough estimate in the non-magnetic case turns out to be sharp for constant magnetic fields. Our second result is an alternative proof of the sharpness of \eqref{eq:riesz_means_bound_ML_known} based on estimates of eigenvalues of the disk. Like Frank, we will show that \eqref{eq:riesz_means_bound_ML_known} is sharp for any given domain $\Omega$ and any $\gamma\geq 0$ (Theorem \ref{thm:magnetic_polya_general_set} and Corollary \ref{cor:magnetic_disk_trace}). 

It should also be noted that the situation of constant magnetic fields differs from another important special case, the case of $\delta$-like magnetic fields which are induced by the so-called Aharonov-Bohm potentials. Whereas P\'olya's bound is violated for constant magnetic fields, it has recently been shown that P\'olya's bound continues to hold for centered Aharonov-Bohm potentials on disks, see Filonov et al.\ \cite{Filonov2023}.

The paper is structured as follows. In Section \ref{sec:statement_of_results}, we review the eigenvalue problem of the magnetic Dirichlet Laplacian with constant magnetic field on disks and state our two main results. In Section \ref{sec:proof_thm2_1} and \ref{sec:proof_thm2_2}, we give the proofs to our main results.

\section{Preliminaries and statement of results} \label{sec:statement_of_results}

To present our results, we need to briefly summarize how the eigenvalue problem for the magnetic Dirichlet Laplacian with constant magnetic field on disks is solved. It is closely related to the eigenvalue problem of the harmonic oscillator confined to disks, see e.g.\ \cite{Son2014} and \cite{Valentim2019}. We follow the outline of the latter reference and modify it to fit the eigenvalue problem of the magnetic Laplacian on disks.

\subsection{Explicit solution of the eigenvalue problem on disks}

Let $D_R$ denote a disk of radius $R>0$. The eigenvalue problem on $D_R$ can be explicitly solved in terms of zeros of special functions. Since the eigenvalues are invariant under rigid transformations of $\Omega$, we may assume the disk $D_R$ is centered at the origin. In polar coordinates, the strong eigenvalue equation turns into
\begin{align}
\left[ -\frac{\partial^2}{ \partial r^2} - \frac{1}{r} \frac{\partial}{ \partial r} - \frac{1}{r^2} \frac{\partial^2}{ \partial \varphi^2} - i B \frac{\partial}{ \partial \varphi} + \frac{B^2}{4}r^2  \right] \psi(r, \varphi) & = \lambda \psi(r, \varphi), \nonumber
\end{align}
where $ r\in(0, R), \, \varphi \in[0, 2\pi), $ while the boundary condition turns into
\begin{align}
 \psi(R, \varphi) &= 0 \nonumber
\end{align}
for any $\varphi \in [0, 2\pi)$. We see that the magnetic Laplacian $(-i\nabla +A)^2$ commutes with the angular momentum operator $L_z = -i \frac{\partial}{ \partial \varphi}$, so the operators have a shared eigenbasis. The corresponding eigenfunctions are given by
\begin{align}
 \psi(r, \varphi) = Z(r) e^{i l \varphi},   \qquad\qquad r\in(0, R) , \quad \varphi \in [0, 2\pi ), \nonumber
\end{align}
where $l\in\mathbb{Z}$ and $Z(r)$ solves the one-dimensional differential equation
\begin{align}\label{eq:magnetic_disk_radial_diffeq}
\left[ \frac{d^2}{ d r^2} +\frac{1}{r} \frac{d}{ d r} - \left( \frac{B^2}{4}r^2 + (B l   - \lambda) + \frac{l^2}{r^2} \right)  \right] Z(r) & = 0, \qquad r\in(0, R),\\ 
\label{eq:magnetic_disk_radial_diffeq_bndcond} Z(R)&=0.
\end{align}
Note that the radial function $Z$ must also have a certain regularity at $r=0$ to ensure that $\psi$ is an admissible eigenfunction, i.e.\ $\psi\in H_0^1(D_R)$.

We assume from now on that $B>0$ since \eqref{eq:magnetic_disk_radial_diffeq} is invariant under simultaneous change of the signs of $B$ and $l$. The solution to \eqref{eq:magnetic_disk_radial_diffeq} can then be written as 
\begin{align}
Z(r) = e^{-\frac{s}{2}} s^\frac{|l|}{2} Y(s), \qquad s = \frac{B}{2} r^2. \nonumber
\end{align}
Here, $Y(s)$ solves
\begin{align}
\left[ s \frac{d^2}{ds^2} + (|l|+1 -s)\frac{d}{ds} - \frac{1}{2} \left( l+ |l|+1- \frac{ \lambda}{B}    \right) \right] Y(s) = 0 \nonumber
\end{align}
which is an instance of Kummer's equation
\begin{align}
\left[ z \frac{d^2}{dz^2} + (b -z)\frac{d}{dz} - a \right] y(z) = 0, \qquad a,b,z \in\mathbb{C},\; b\neq 0,-1,-2,... \; . \nonumber
\end{align}
Solutions of Kummer's equation are usually given as linear combinations of Kummer's confluent hypergeometric function $M(a,b,z)$ (also often denoted by ${_1}F_1(a,b,z)$) and Tricomi's confluent hypergeometric function $U(a,b,z)$. Of those two functions, the former is regular at $z = 0$ and eventually leads to eigenfunctions $\psi \in H^1_0(D_R)$ whereas the latter has a singularity at $z=0$ that eventually leads to $|\psi|^2 $ and/or $ |\nabla \psi|^2 $ not being locally integrable around the origin. Therefore, only $M(a,b,z)$ is relevant to us. Recall that Kummer's function $M(a,b,z)$ is entire in $a$ and $z$ and given by the power series
\begin{align}
M(a,b,z) = \sum\limits_{k=0}^\infty \frac{(a)_k}{(b)_k k!} z^k, \qquad b \neq 0, -1,-2,... \nonumber
\end{align}
where 
\begin{align}
(c)_k = \begin{cases} 1 & \text{for } k =0, \\
 c \cdot (c+1) \cdot ... \cdot (c+k-1) & \text{for } k \geq 1,  \end{cases} \qquad c \in \mathbb{C}, \; k\in\mathbb{N}_0, \nonumber
\end{align}
denotes the Pochhammer symbol (for references on confluent hypergeometric functions see \cite{Buchholz1953}, \cite[Vol. I, Chapter VI]{Bateman1953}, \cite[Chapter 13]{Abramowitz1972} and \cite[\href{https://dlmf.nist.gov/13}{Chapter 13}]{NIST_DLMF}). In order to be an eigenvalue, the yet unspecified parameter $\lambda$ must be chosen in such a way that $Z(r)$ satisfies the Dirichlet boundary condition \eqref{eq:magnetic_disk_radial_diffeq_bndcond}. This means that $\lambda$ must be a solution to the implicit equation
\begin{align}
M \left(\frac{1}{2} \left( l+ |l|+1- \frac{ \lambda}{B}    \right)  ,|l|+1, \frac{BR^2}{2}  \right) = 0. \label{eq:magnetic_disk_eigval_transcendental}
\end{align}

It is known that the function $a \mapsto M(a,b,z)$, where $a\in\mathbb{R}$, $b>0$, $z\geq 0$, has no non-negative roots and an infinite number of simple, negative roots $a_m(b,z)$ with accumulation point at $-\infty$ (see e.g.\ \cite[\S 17]{Buchholz1953}). Let us assume the roots $a_m(b,z)$ are sorted in decreasing order. We then denote the solutions of the implicit equation \eqref{eq:magnetic_disk_eigval_transcendental} by $\lambda_{m,l}(B)$ according to the relation 
\begin{align}
\lambda_{m,l}(B) = \left(l+|l|+1 -2a_m\left(|l|+1,\frac{BR^2}{2}\right) \right) B. \label{eq:eigenvalue_lambda_a_m_relation}
\end{align}
The eigenvalues $\lambda_{m,l}(B)$ then satisfy 
\begin{align}
\lambda_{1,l}(B) \leq \lambda_{2,l}(B) \leq ...  \nonumber
\end{align}
By the implicit function theorem, all maps $z\mapsto a_m(b,z)$ are real analytic functions of $z$. This implies that all maps $B \mapsto \lambda_{m,l}(B)$, $n\in\mathbb{N}$, $l \in\mathbb{Z}$, are real analytic functions of $B$.

Finally, the eigenfunctions
\begin{align}
\psi_{m,l}(r, \varphi) = Z_{m,l}(r) e^{i l \varphi} , \nonumber
\end{align}
where $Z_{m,l}$ are the solutions of the radial equation \eqref{eq:magnetic_disk_radial_diffeq} with $\lambda = \lambda_{m,l}(B)$, $m\in\mathbb{N}$, $l\in\mathbb{Z}$, form a complete orthonormal basis of $L^2(D_R)$. This is a consequence of $ \lbrace Z_{m,l} \rbrace_{m\in\mathbb{N}}$ forming a complete orthonormal basis of $L^2((0,R), r\, dr)$ which follows from classical Sturm-Liouville theory (for more details see \cite{Son2014}). The set $\lbrace \lambda_{m,l}(B) \, : \, m\in\mathbb{N}, l\in\mathbb{Z} \rbrace$ gained by solving \eqref{eq:magnetic_disk_eigval_transcendental} is thus indeed the complete set of all eigenvalues. We have found a complete description of all eigenvalue branches of the magnetic Dirichlet Laplacian with constant magnetic field on $D_R$. Ordering $\lbrace \lambda_{m,l}(B) \, : \, m\in\mathbb{N}, l\in\mathbb{Z} \rbrace$ by magnitude and relabeling results in the set of sorted eigenvalues $\lbrace \lambda_n(D_R,B) \rbrace_{n\in\mathbb{N}}$.

\subsection{Asymptotics for eigenvalue branches of the disk}

Figure \ref{fig:diamagnetism} displays the lowest few eigenvalues $\lambda_n(D_R,B)$ over $B$ for a disk of unit area. It suggests that the eigenvalue branches have an asymptotically linear behavior as $B \rightarrow +\infty$ and there appear sectors that only certain branches can populate. A detailed investigation of the eigenvalue branches $\lambda_{m,l}(B)$ has been carried out in the doctoral thesis of Son \cite[Theorem 3.3.4]{Son2014} where it is proven that 
\begin{align}
\lim_{B \rightarrow +\infty} \frac{\lambda_{m,l}(B)}{B} = l+ |l|+1 + 2(m-1). \nonumber
\end{align}

Son posed the question for the asymptotic expression of the remainder $\lambda_{m,l}(B)/B - (l+ |l|+1 + 2(m-1))$ which is equivalent to asking for the next terms in the asymptotic expansion of $\lambda_{m,l}(B)$. In the case of the ground state (which can be identified with the branch $\lambda_{1,0}(B)$), this question was already addressed by Helffer and Morame \cite[Proposition 4.4]{Helffer2001} in 2001. However, their asymptotic expression is incorrect and was also mistakenly cited by Ekholm, Kova{\v{r}}{\'{\i}}k and Portmann \cite{Ekholm2016} and Fournais and Helffer \cite{Fournais2010}. Helffer and Persson Sundqvist \cite{Helffer2017} corrected the erroneous asymptotic from \cite{Helffer2001} by showing an asymptotic in the semi-classical setting via construction of trial states. Theorem 5.1 of \cite{Helffer2017} asserts that if $D_R$ is a disk of radius $R>0$, then for any $n\in\mathbb{N}$
\begin{align}
\frac{ \lambda_n(D_R, B)}{B}  - 1 = \frac{2}{(n-1)!}\left(\frac{BR^2}{2}\right)^{n} e^{-BR^2/2}(1+O(B^{-1})) \qquad \text{as } B\rightarrow +\infty. \nonumber
\end{align}
Barbaroux et al.\ \cite{Barbaroux2021} gave a generalization of this result for non-constant magnetic fields and fairly generic domains.

This however does not fully answer the question about the second term in the asymptotic expansion of the eigenvalue branches $\lambda_{m,l}(B)$ for all $m\in\mathbb{N}$, $l\in\mathbb{Z}$. The construction of trial states only allows treating eigenvalues sorted in increasing order. As we will see with Corollary \ref{cor:magnetic_disk_lambda_lowest}, the sorted eigenvalues $\lambda_n(\Omega, B)$ coincide with the eigenvalue branches $\lambda_{1,-(n-1)}(B)$ for large enough $B$, so the result by Helffer and Persson Sundqvist only covers the case $m=1$ and $l\leq 0$.

\begin{figure}[p]
\centering
\begin{subfigure}{0.9\textwidth}
\begin{center}
\includegraphics[scale=0.4,trim={0 0.6cm 0cm 1.0cm},clip]{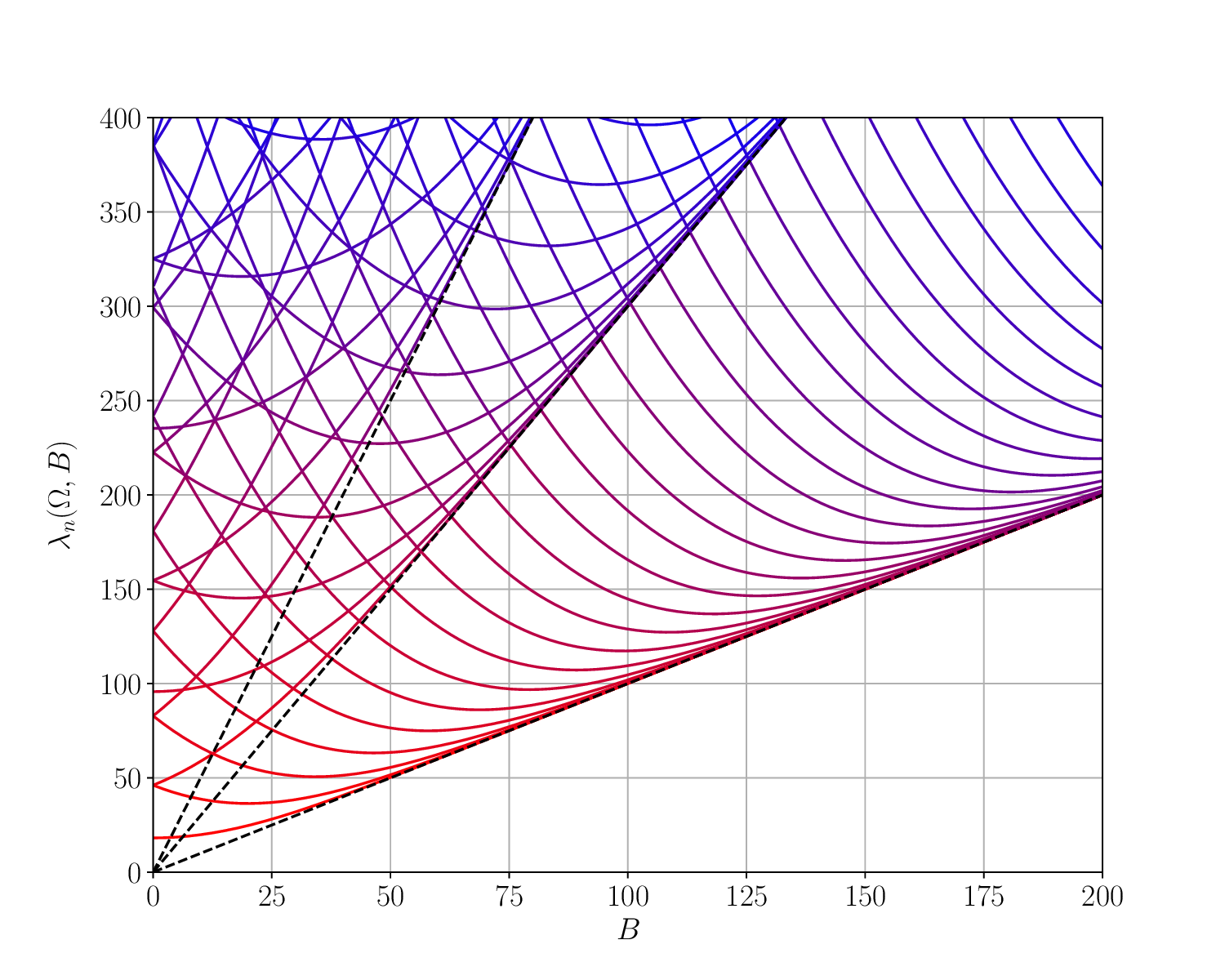}\hspace*{0.3cm}
\end{center}
\caption{} \label{fig:diamagnetism}
\end{subfigure}

\begin{subfigure}{0.9\textwidth}
\begin{center}
\includegraphics[scale=0.4,trim={0 0.6cm 0cm 1.0cm},clip]{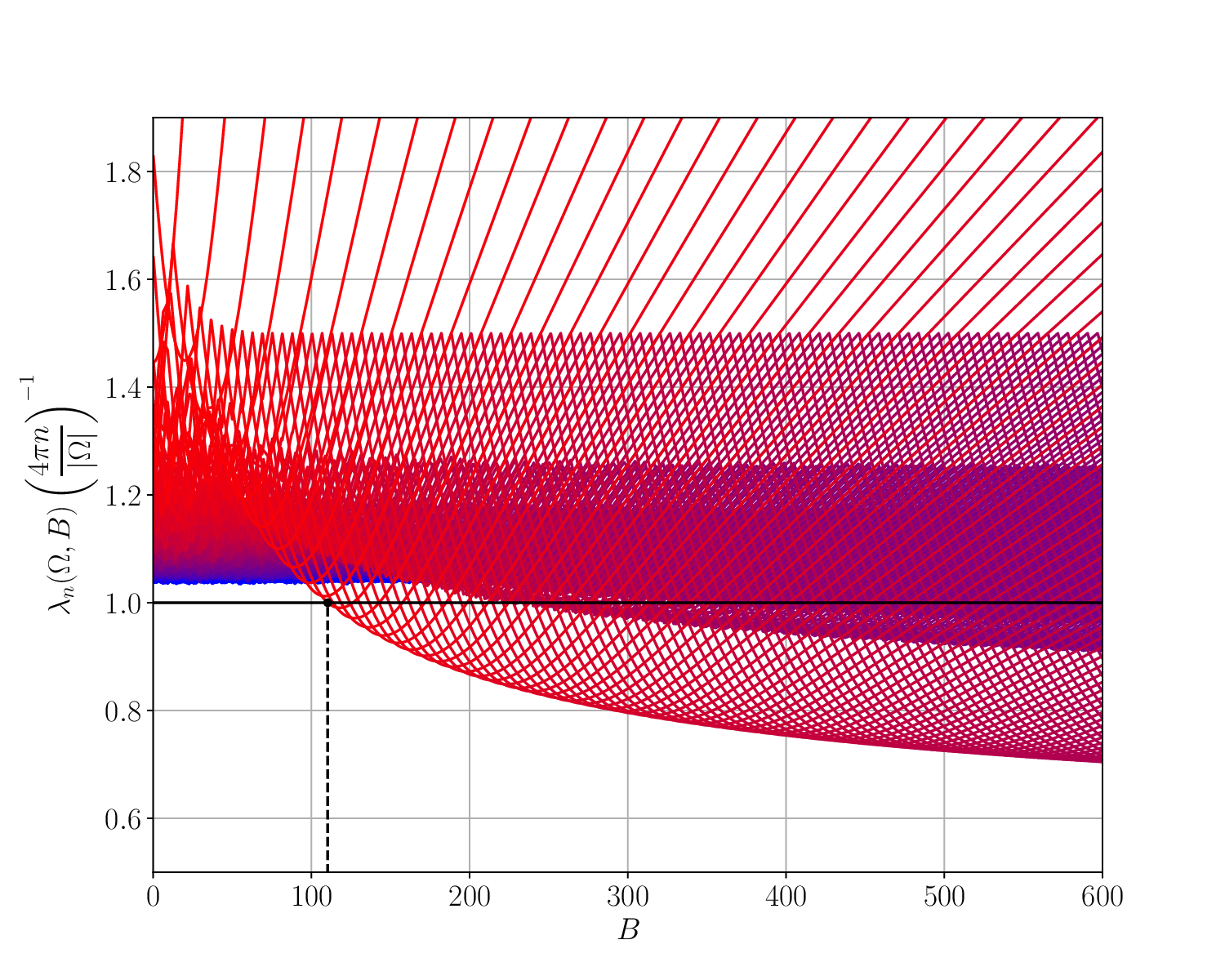} \hspace*{0.2cm}
\end{center}
\caption{} \label{fig:polya}
\end{subfigure}

\caption{Eigenvalues of the magnetic Dirichlet Laplacian on a disk of radius of $R=1/\sqrt{\pi}$ (and hence area $|D_R|=1$). Color of the curves fades from red to blue as $n$ increases. (a) Low eigenvalues $\lambda_n(D_R,B)$ over the field strength $B$. The dashed lines indicate the lines $\lambda = B$, $3B$ and $5B$. (b) The first 500 eigenvalues divided by $4\pi n / |D_R|$ plotted over the field strength $B$. The value $1.0$ associated with the lower bound from the non-magnetic P\'olya conjecture is crossed for the first time at $B_{\mathrm{crit}} \approx 110.335$ by $\lambda_{11}(D_R, B)$. This point is marked by a black dot and a dashed line.  } 
\end{figure}

Our first result gives the second term in the asymptotic expansion for all eigenvalue branches $\lambda_{m,l}(B)$ and therefore generalizes the result by Helffer and Persson Sundqvist.

\begin{theorem} \label{thm:eigenvalue_asymptotics_disk}
For any fixed $m\in\mathbb{N}$, $l\in\mathbb{Z}$  and $R>0$ holds
\begin{align}
\frac{ \lambda_{m,l}(B)}{B} >  l+ |l|+1 + 2(m-1) \qquad \text{for all } B>0 \nonumber
\end{align}
and
\begin{align}
&\frac{ \lambda_{m,l}(B)}{B}  - (l+ |l|+1+2(m-1)) \nonumber \\ 
&\qquad \qquad =  \frac{ 2}{ \Gamma(|l|+m)\Gamma(m)} \left(\frac{BR^2}{2}\right)^{|l|+1+2(m-1)} e^{-BR^2/2}(1+O(B^{-1}))  \nonumber
\end{align}
as $B\rightarrow +\infty$.
\end{theorem}

We give the proof of Theorem \ref{thm:eigenvalue_asymptotics_disk} in Section \ref{sec:proof_thm2_1}. We remark that the proof could be extended to gain more orders in the asymptotic expansion.

\subsection{Spectral inequalities and sharpness}

For the constant magnetic field, Erd\H{o}s, Loss and Vougalter \cite{Erdoes2000} proved that 
\begin{align} \label{eq:liyau_bound_ML}
\sum_{k=1}^n \lambda_k(\Omega,B)\geq \frac{2\pi n^2}{|\Omega|},  \qquad  n\in\mathbb{N},
\end{align}
for any $B\in\mathbb{R}$. Earlier, Li and Yau \cite{Li1983} had shown the inequality in the non-magnetic case. It was then realized that the Li-Yau bound is equivalent to Berezin's bound by Legendre transformation, see e.g.\ \cite[Chapter 3.5]{Frank2022}. Similarly, \eqref{eq:liyau_bound_ML} is equivalent to \eqref{eq:riesz_means_bound_ML_known} for $\gamma=1$.

By the same reasoning as without magnetic field, the result by Erd\H{o}s, Loss and Vougalter implies \eqref{eq:riesz_means_bound_ML_known} with the excess factors $R_\gamma$ defined in the introduction and in particular 
\begin{align}\label{eq:lambda_n_bound_ML_known}
\lambda_n(\Omega,B)\geq \frac{1}{2} \cdot \frac{4\pi n}{|\Omega|},  \qquad  n\in\mathbb{N}, B\in\mathbb{R}, 
\end{align}
for any domain $\Omega$. Compared to P\'olya's conjectured bound \eqref{eq:polya_bound} for the non-magnetic case there appears an excess factor of $1/2$. 

Our second main result is a proof of the sharpness of \eqref{eq:lambda_n_bound_ML_known} for any domain $\Omega$ when $B$ is allowed to become large. That is, the excess factor of $1/2$ in \eqref{eq:lambda_n_bound_ML_known} cannot be improved.

\begin{theorem} \label{thm:magnetic_polya_general_set}
Let $\Omega \subset \mathbb{R}^2$ be a domain of finite measure. Then for any $\eta > 0$ there exists some $B_{\eta}>0$ such that for any $B\geq B_{\eta}$ one has at least one eigenvalue $\lambda_n(\Omega, B)$ with
\begin{align}
\frac{1}{2} \leq \lambda_n(\Omega, B) \left(\frac{4\pi n}{|\Omega|}  \right)^{-1} \leq \frac{1}{2}+\eta . \nonumber
\end{align}
\end{theorem}

\noindent Note that the lower bound is just \eqref{eq:lambda_n_bound_ML_known} which we do not reprove. We prove the upper bound to show sharpness of the excess factor. We do this first in the case of the domain being a disk by an explicit calculation. Then, we extend the result to disjoint unions of disks and finally to any open domain $\Omega$ of finite measure by a domain inclusion argument. 

As a consequence of the Theorem \ref{thm:magnetic_polya_general_set}, none of the excess factors $R_\gamma$, $\gamma\geq 0$, in \eqref{eq:riesz_means_bound_ML_known} can be improved for any fixed domain. 

\begin{corollary} \label{cor:magnetic_disk_trace}
Let $\Omega \subset \mathbb{R}^2$ be a domain of finite measure and $\gamma \geq 0$. Then for any $\eta > 0$, there exists some $B_\eta>0$ such that for any $B\geq B_\eta$ one finds some $\lambda>0$ with
\begin{align} \label{eq:cor_magnetic_disk_trace}
R_\gamma -\eta \leq \frac{\mathrm{tr} (H_B^\Omega - \lambda)_-^\gamma }{  L_{\gamma,2}^{\mathrm{cl}}  |\Omega| \lambda^{1+\gamma} }  \leq R_\gamma. 
\end{align}
\end{corollary}
\noindent This is Frank's sharpness result \cite[Theorem 3.6]{Frank2009b} that was already mentioned in the introduction. The proof of the corollary will be given right after the proof of Theorem \ref{thm:magnetic_polya_general_set} in Section \ref{sec:proof_thm2_2}. 

Figure \ref{fig:polya} illustrates Theorem \ref{thm:magnetic_polya_general_set} for the disk. It shows the normalized eigenvalues $\lambda_n(\Omega, B)(4\pi n/|\Omega|)^{-1}$, i.e.\ the eigenvalues normalized by the right-hand-side in P\'olya's inequality resp. Weyl's law. One observes that once the field strength surpasses a critical value, the normalized eigenvalues start to drop below one, indicated by a black, horizontal line. This is exactly  when the sorted eigenvalues $\lambda_n(\Omega, B)$ fall below the non-magnetic P\'olya bound of $4\pi n/|\Omega|$ and the non-magnetic P\'olya bound is violated. Theorem \ref{thm:magnetic_polya_general_set} asserts that any domain $\Omega$ features this kind of behavior and that 
\begin{align} \label{eq:lim_polya_excess_factor}
\lim_{B\rightarrow \infty} \; \inf_{n\in\mathbb{N}} \lambda_n(\Omega, B) \left(\frac{4\pi n}{|\Omega|} \right)^{-1} = \frac{1}{2}. 
\end{align}
Graphically, this means that the lower envelope of the curves in Figure \ref{fig:polya} converges to $1/2$ as $B\rightarrow \infty$. 

With some more work, our proof of Theorem \ref{thm:magnetic_polya_general_set} allows an estimate on the convergence speed of \eqref{eq:lim_polya_excess_factor} for the disk. For any $\rho < 1/2$, it is possible to show an upper bound of the form
\begin{align} \label{eq:polya_excess_factor_landau}
\inf_{n\in\mathbb{N}} \lambda_n(D_R, B) \left(\frac{4\pi n}{|D_R|} \right)^{-1} \leq \frac{1}{2} + C_R B^{-\rho} 
\end{align}
for $B$ large enough, where $C_R $ solely depends on $R$. We do not give a proof for \eqref{eq:polya_excess_factor_landau} in this article, but we plan to provide such improved convergence results in a future article.

\section{Proof of Theorem \ref{thm:eigenvalue_asymptotics_disk}} \label{sec:proof_thm2_1}

For completeness, we begin with a proof of the fact that
\begin{align}
\lim_{B \rightarrow +\infty} \frac{\lambda_{m,l}(B)}{B} = l+ |l|+1 + 2(m-1).  \nonumber
\end{align}
After that, we compute the asymptotic expression of the remainder term. 

Equation \eqref{eq:magnetic_disk_eigval_transcendental} relates eigenvalues for the magnetic Dirichlet Laplacian on the disk with roots of Kummer's function. Statements about roots of $a\mapsto M(a,b,z)$ directly translate to properties of eigenvalues. We must therefore closely analyze Kummer's function and its roots. The arguments of Kummer's function in \eqref{eq:magnetic_disk_eigval_transcendental} are always real-valued, so the arguments $a,b$ and $z$ of Kummer's function will be considered real numbers in the following. 

The existence of a limit of $\lambda_{m,l}(B)/B$ can be established by monotone convergence. For $\lambda \leq (l+ |l|+1) B$ we have
\begin{align}
\frac{1}{2} \left( l+ |l|+1- \frac{ \lambda}{B}    \right)\geq 0  \nonumber
\end{align}
and thus 
\begin{align}
&M \left(\frac{1}{2} \left( l+ |l|+1- \frac{ \lambda}{B}    \right),|l|+1, \frac{BR^2}{2}  \right) \\
=\; & 1 + \sum\limits_{k=1}^\infty \frac{\left(\frac{1}{2} \left( l+ |l|+1- \frac{ \lambda}{B}    \right) \right)_k}{(|l|+1)_k k!} \left( \frac{BR^2}{2}  \right)^k \geq 1,   \nonumber
\end{align}
as all summands in the series are non-negative. Therefore, \eqref{eq:magnetic_disk_eigval_transcendental} cannot be satisfied and $\lambda$ cannot be an eigenvalue. This means that any eigenvalue $\lambda_{m,l}(B)$ must fulfill
\begin{align}
\frac{\lambda_{m,l}(B)}{B}> l+ |l|+1.
\end{align}
Note that the roots $a_m(b,z)$ of $a \mapsto M(a,b,z)$ are strictly increasing with $z>0$ for any fixed $b>0$ and $m\in\mathbb{N}$, see e.g.\ \cite[\S 17]{Buchholz1953}. According to \eqref{eq:eigenvalue_lambda_a_m_relation}, this implies that all maps $B \mapsto \lambda_{m,l}(B)/B$ are strictly decreasing. By monotone convergence, the limit $\lim_{B \rightarrow +\infty} \lambda_{m,l}(B)/B$ exists and is bounded from below by $(l+ |l|+1)$. 

Let us now determine this limit. Let $b=|l|+1>0$ be fixed. For any $m\in \mathbb{N}_0$ we have 
\begin{align}
M(-m,b,z) = \sum\limits_{k=0}^\infty \frac{(-m)_k}{(b)_k k!} z^k = \sum\limits_{k=0}^{m} \frac{(-m)_k}{(b)_k k!} z^k, 
\end{align}
since $(-m)_k = 0$ for $k>m$. The map $z \mapsto M(-m,b,z)$ is a (scaled Laguerre) polynomial in $z$ of degree $m$ and for real $z$ we have
\begin{align}
M(-m,b,z) = \frac{(-m)_m}{(b)_m m!} z^m (1+O(z^{-1})) =\frac{(-1)^m}{(b)_m} z^m (1+O(z^{-1})) \qquad \text{as } z\rightarrow +\infty.
\end{align}
Thus, if $z^*>0$ is chosen large enough, one can enforce that $\text{sgn}(M(-m,b,z)) = (-1)^m$ for all $z>z^*$. Let now $m' \in\mathbb{N}$. One can choose $z^*>0$ large enough such that $\text{sgn}(M(-m,b,z^*)) = (-1)^m$ for all integers $m$ with $0 \leq m\leq m'$. As $a \mapsto M(a,b,z)$ is continuous for fixed $z>z^*$, it must have at least one root in each of the intervals $(-m,-(m-1))$, $1 \leq m\leq m'$, by the intermediate value theorem. However, it is actually true that each of these intervals contains exactly one root. This is a consequence of an intimate connection between roots of $a\mapsto M(a,b,z)$ and roots of $z \mapsto M(a,b,z)$ which is the content of the next lemma. We will see that $a_m(b,z)$ must be bounded from above by $-(m-1)$.

To state the lemma, let us introduce some new notation. First, since $b$ is fixed, we suppress the parameter $b$ in the following and just write $a_m(z)$ for the negative roots of $a\mapsto M(a,b,z)$ instead of $a_m(b,z)$. Recall that for every fixed $z>0$ the roots $a_m(z)$ are sorted in decreasing order. The function $ z \mapsto M(a,b,z)$ for $a<0$, $b> 0$, has exactly $\lceil -a \rceil$ real, positive roots, see \cite[\S 9]{Tricomi1955}, \cite[\S 17]{Buchholz1953} or \cite[\href{https://dlmf.nist.gov/13.9}{Section 13.9}]{NIST_DLMF}. Here, $\lceil \, . \, \rceil$ denotes the ceiling function. We sort the roots in increasing order and denote the $m$-th root by $z_m(a)$. As an illustration, Figure \ref{fig:kummerM_a_m_z_m} shows the first few roots $a_m(z)$ and $z_m(a)$ for the case $b=1$. We will now proceed to prove the following lemma.

\begin{lemma} \label{lem:am_zm_roots}
Let $a_m(z)$, $z_m(a)$ be defined as above. Then for all $m\in\mathbb{N}$
\begin{align}
a_m(z)<-(m-1) \qquad \text{for all } z \in (0, \infty).  \nonumber
\end{align}
Moreover, $a_m: (0, \infty) \rightarrow (-\infty, -(m-1))$ is invertible and its inverse is given by $z_m: (-\infty, -(m-1)) \rightarrow (0, \infty)$.
\end{lemma}

The lemma not only states that $a_m(z)$ can be inverted by some $z_{k}(a)$ but that the index $k$ is always equal to $m$. Knowing this, the upper bound on the root $a_m(z)$ is owed to the fact that there are only $\lceil -a \rceil$ real, positive roots $z_m(a)$ for fixed $a$. 

Lemma \ref{lem:am_zm_roots} appeared as part of Proposition 4.3.12 in the thesis of Son \cite{Son2014}. The importance of the lemma is two-fold. For one, it guarantees the upper bound $a_m(z)<-(m-1)$ that we need in this section. Secondly, we will make use of the characterization of the inverse of $a_m(z)$ in the proof of Theorem \ref{thm:magnetic_polya_general_set}, see Corollary \ref{cor:magnetic_disk_lambda_explicit}. For convenience of the reader, we give a self-contained proof of the lemma. 

\begin{proof}
First, note that all roots $a_m(z)$ and $z_m(a)$, $m\in\mathbb{N}$, are simple. For the roots $a_m(z)$, we already mentioned this fact before. The roots $z_m(a)$ are simple as well. This is because, if one of them were a multiple root, repeated use of Kummer's equation would imply that this particular root is of infinite order. But viewed as a complex function, Kummer's function is entire in $z$. A root of infinite order would imply that Kummer's function is identical to the zero function which is a clear contradiction.

\begin{figure}
\centering
\begin{center}
\includegraphics[scale=0.35,trim={0.0cm 0.5cm 1.5cm 1.5cm},clip]{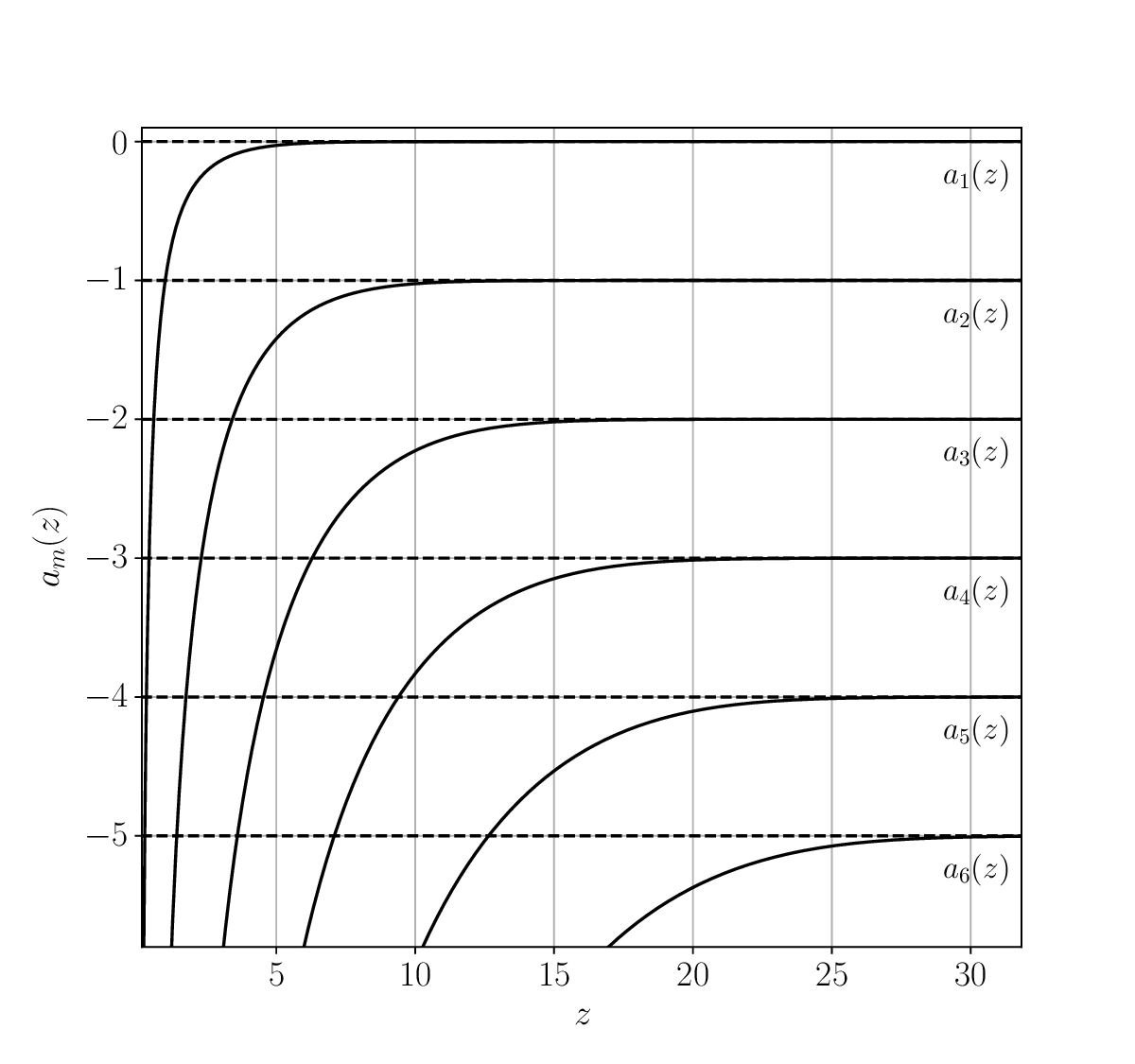}
\includegraphics[scale=0.35,trim={0.5cm 0.5cm 1.5cm 1.5cm},clip]{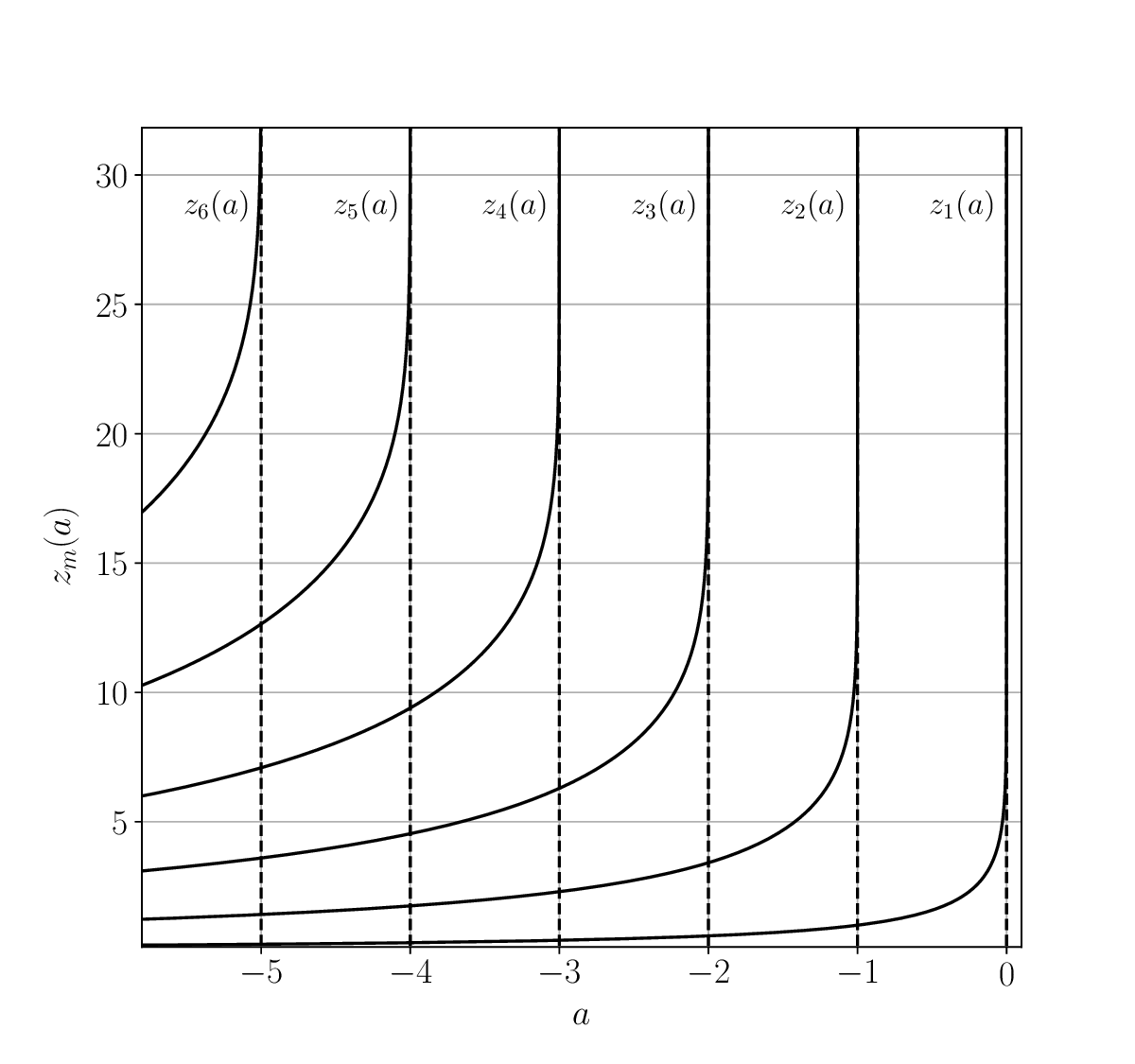}
\end{center}
\caption{The first few roots $a_m(z)$ and $z_m(a)$ for $b=1$.} \label{fig:kummerM_a_m_z_m}
\end{figure}

For fixed $a<0$ there are exactly $\lceil -a \rceil$ roots $z_m(a)$, which means that the $m$-th root $z_m(a)$ only exists for $a\in (-\infty, -(m-1))$. By the implicit function theorem, all maps $a_m$ and $z_m$, defined on $(0,\infty)$ resp. $(-\infty, -(m-1))$ are smooth functions. For now, we only know that the image of $a_m$ is contained in $(-\infty, 0)$, because we do not have any constraints yet where the roots $a_m(z)$ lie other than that they are negative. We know $a_m$ is injective, since $a_m(z)$ is strictly increasing with $z$. We need to show that the image of $a_m$ is $(-\infty, -(m-1))$ and inverted by $z_m$. 

Let $a \in (-\infty, 0)$, $z\in(0, \infty)$. Note that $a$ is a root of $a'\mapsto M(a',b,z)$ if and only if $z$ is a root of $z'\mapsto M(a,b,z')$. This means that $a=a_m(z)$ only if $z=z_{k}(a)$ for some $k\in\mathbb{N}$, $1\leq k \leq \lceil -a \rceil$, and conversely $z=z_{k}(a)$ only if $a=a_m(z)$ for some $m\in\mathbb{N}$. We wrote $k$ because we do not know yet what the connection between $m$ and $k$ is. We now show that $m=k$ always holds. In other words, we show that
\begin{align} 
\text{if } a \in [ -(m-1),0), \qquad &\text{then} \qquad a=a_m(z) \text{ has no solution } z\in(0,\infty)  \nonumber \\
\text{and if }  a \in (-\infty, -(m-1)), \qquad &\text{then} \qquad a=a_m(z) \text{ if and only if } z=z_{m}(a).  \nonumber
\end{align}
 Once this is shown, we have that the image of $a_m$ is contained in $(-\infty, -(m-1))$ and $z_m:(-\infty, -(m-1)) \rightarrow (0, \infty)$ is the inverse of $a_m: (0, \infty) \rightarrow (-\infty, -(m-1))$ which is the assertion of the lemma.

We carry out an induction over $m$. For the start of the induction, let $a\in(-\infty, 0)$ be arbitrary and suppose $a=a_1(z_{k}(a))$ for some integer $1< k \leq \lceil -a \rceil$. There must exist some integer $m \in\mathbb{N}$ such that $a=a_{m}(z_1(a))$. But then the sortings $a_1(z)>a_2(z)> ...$ and $z_1(a)<z_2(a)< ...$ and the strict monotonicity of $a_1(z)$ with respect to $z$ yields the contradiction
\begin{align}
a= a_{m}(z_1(a)) \leq a_1(z_1(a))< a_1(z_{k }(a))=a.  \nonumber
\end{align}
Therefore, for any $a\in(-\infty, 0)$, $a=a_1(z)$ if and only if $z = z_1(a)$.

Let us now assume that there exists an $m'\in\mathbb{N}$ such that for all $m$ with $1\leq m \leq m'$ we have 
\begin{align} 
\text{if } a \in [ -(m-1),0), \qquad &\text{then} \qquad a=a_m(z) \text{ has no solution } z\in(0,\infty)  \nonumber \\
\text{and if }  a \in (-\infty, -(m-1)), \qquad &\text{then} \qquad a=a_m(z) \text{ if and only if } z=z_{m}(a).  \nonumber
\end{align}
We need to show that 
\begin{align} 
\text{if } a \in [ -(({m'+1})-1),0), \quad &\text{then} \quad a=a_{m'+1}(z) \text{ has no solution } z\in(0,\infty)  \nonumber \\
\text{and if } a\in(-\infty, -(({m'+1})-1)), \quad &\text{then} \quad a=a_{m'+1}(z) \text{ if and only if } z=z_{m'+1}(a).  \nonumber 
\end{align}

Let $a\in(-\infty, 0)$ be arbitrary and assume that we have a $z\in(0,\infty)$ with $a = a_{m'+1}(z)$. We know that $z=z_{k}(a)$ for some integer $k$ satisfying $1\leq k \leq \lceil -a \rceil$. First of all, if $k\leq m'$, then by induction assumption $a = a_{k}(z_{k}(a))$. But this leads to the contradiction
\begin{align}
a = a_{m'+1}(z_{k}(a)) < a_{k}(z_{k}(a)) = a,   \nonumber
\end{align}
because $k < m'+1 $. Thus, we must have $k\geq  m' + 1$. 

We can now discuss the solvability of $a=a_{m'+1}(z)$ case-by-case.

\textit{Case 1:} If $a \in [ -(({m'+1})-1),0) = [-m',0)$, then $1\leq k \leq \lceil -a \rceil$ implies $k\leq m'$ which contradicts $k\geq  m' + 1$. Therefore, there is no solution $z \in (0,\infty) $ to $a=a_{m'+1}(z)$.

\textit{Case 2:} Consider $a\in(-\infty, -(({m'+1})-1)) = (-\infty, -m')$. 

\; \textit{Case 2.1:} If $a\in[-m'-1, -m')$, then $m'+1 \leq k\leq \lceil -a \rceil$ implies $k=m'+1$. Thus, $a=a_{m'+1}(z)$ if and only if $z=z_{m'+1}(a)$. 

\; \textit{Case 2.2:} If $a\in(-\infty, -m'-1)$, then more than $m'+1$ roots in $z$ exist, namely $z_1(a)$, ..., $z_{\lceil -a \rceil}(a)$, so that $k$ could be larger than $m'+1$. 
Suppose $k> m'+1$. There must exist some integer $m \in \mathbb{N}$ such that $a=a_{m}(z_{m'+1}(a))$. If $m < m' +1$, then by the induction assumption $a=a_{m}(z_{m}(a))$ which leads to the contradiction
$$a=a_{m}(z_{m}(a)) < a_{m}(z_{m'+1}(a))= a.$$
If $m\geq  m' +1$, then 
\begin{align}
a = a_{m}(z_{m'+1}(a)) \leq a_{m'+1}(z_{m'+1}(a)) < a_{m'+1}(z_{k}(a)) = a  \nonumber
\end{align}
which is again a contradiction. Hence, $k = m'+1$ and for all $a\in(-\infty, -(({m'+1})-1))$ holds $a=a_{m'+1}(z)$ only if $z=z_{m'+1}(a)$. This concludes the induction and the proof.
\end{proof}

We return to the proof of Theorem \ref{thm:eigenvalue_asymptotics_disk}. By Lemma \ref{lem:am_zm_roots} we have $a_m(z)<-(m-1)$ for all $z \in (0, \infty)$. Since $a_m:(0, \infty) \rightarrow (-\infty, -(m-1)) $ is strictly increasing and bijective, we must have $a_m(z) \rightarrow -(m-1) $ for $z \rightarrow +\infty$. Together with \eqref{eq:eigenvalue_lambda_a_m_relation}, these two facts show that 
\begin{align}
\frac{ \lambda_{m,l}(B)}{B} >  l+ |l|+1 + 2(m-1) \qquad \text{for all } B>0  \nonumber
\end{align}
and
\begin{align}
\lim_{B \rightarrow +\infty} \frac{\lambda_{m,l}(B)}{B} = l+ |l|+1 + 2(m-1).  \nonumber
\end{align}
The lower bound is the first statement of Theorem \ref{thm:eigenvalue_asymptotics_disk}.

We can now discuss the asymptotic behaviour of the remainder $ \lambda_{m,l}(B)/B - (l+ |l|+1 + 2(m-1))$. For this, we fix $b=|l|+1$ and set $\varepsilon_m(z) := -(m-1) - a_m(z)$ where - as previously - $a_m(z)$ is the $m$-th negative root of $a\mapsto M(a,b,z)$ for $z> 0$ and fixed $b>0$. By Lemma \ref{lem:am_zm_roots}, $\varepsilon_m(z)$ is always positive. Since $a_m(z) \rightarrow -(m-1) $ for $z \rightarrow +\infty$, we know that $\varepsilon_m(z) = o(1)$ for $z \rightarrow +\infty$. We look for an asymptotic expression of $\varepsilon_m(z)$ as $z\rightarrow +\infty$, since this translates directly into an asymptotic expression for the remainder $ \lambda_{m,l}(B)/B - (l+ |l|+1 + 2(m-1))$ by \eqref{eq:eigenvalue_lambda_a_m_relation}. 

First, we split the series that defines $M(a_m(z), b, z)$ into two parts. If
\begin{align}
p_m(z) &:= \sum\limits_{k=0}^{m-1} \frac{(-(m-1)-\varepsilon_m(z))_k}{(b)_k k!} z^k , \nonumber \\
q_m(z) &:= \sum\limits_{k=m}^{\infty} \frac{(-(m-1)-\varepsilon_m(z))_k}{(b)_k k!} z^k, \nonumber
\end{align}
then
\begin{align} 
p_m(z) + q_m(z) = M(a_m(z), b, z) = 0. \nonumber
\end{align}
The reason why splitting the series at $k=m$ is helpful to gain an asymptotic expression for $\varepsilon_m(z)$ is that $q_m(z)$ contains all terms that have $\varepsilon_m(z)$ as a factor appearing in the Pochhammer symbol in the numerator. We can factor out $\varepsilon_m(z)$ from $q_m(z)$ and write $q_m(z) = \varepsilon_m(z) \tilde{q}_{m}(z)$. Then,
\begin{align}
\varepsilon_m(z) = - \frac{p_m(z)}{\tilde{q}_{m}(z)}  = \left| \frac{p_m(z)}{\tilde{q}_{m}(z)}\right| \nonumber
\end{align}
and the asymptotic behavior of $\varepsilon_m(z) $ is determined by $p_m(z)$ and $\tilde{q}_{m}(z)$. As we will see, both $|p_m(z)|$ and $|\tilde{q}_{m}(z)|$ tend to infinity as $z\rightarrow +\infty$, but at different speed. It remains to investigate their asymptotic behavior as $z\rightarrow +\infty$.

Let us first analyze the asymptotic behavior of $p_m(z)$. For $c \in \mathbb{R}$ and $k \in \mathbb{N}_0$ the function $x \mapsto (c+x)_k $ is a polynomial in $x$ and $(c+x)_k = (c)_k + O(x)$ as $x \rightarrow 0$. Now, because we already know that $\varepsilon_m(z) \rightarrow 0 $ as $z \rightarrow +\infty$, this yields that $(-(m-1)-\varepsilon_m(z))_k = (-(m-1))_k (1+O(\varepsilon_m(z))) = (-(m-1))_k (1+o(1)) $ as $z \rightarrow +\infty$ for any fixed $k \in \mathbb{N}_0$ and $m \in\mathbb{N}$. Therefore, $p_m(z)$ is asymptotically dominated by the term with $k=m-1$, or more precisely
\begin{align} \label{eq:pm_asymptotic_simple}
p_m(z) = \frac{(-(m-1))_{m-1}}{(b)_{m-1} (m-1)!} z^{m-1} (1 + o(1))  = \frac{(-1)^{m-1}}{(b)_{m-1}} z^{m-1} (1 + o(1)) \qquad \text{as } z\rightarrow +\infty.
\end{align}

Let us now take a closer look at $q_m(z)$. Using $(b)_{i+j}=(b)_i (b+i)_j$ for $b \in\mathbb{R}$ and $i,j \in\mathbb{N}_0$ and $n! = (1)_{n}$ for $n \in\mathbb{N}_0$ yields 
\begin{align}
q_m(z)&=\sum\limits_{k=m}^{\infty} \frac{(-(m-1)-\varepsilon_m(z))_k}{(b)_k k!} z^k \nonumber \\
&= (-(m-1)-\varepsilon_m(z))_{m} z^m \sum\limits_{k=m}^{\infty} \frac{(1-\varepsilon_m(z))_{k-m}}{(b)_k k!} z^{k-m} \nonumber \\
&= (-1)^m(1+\varepsilon_m(z))_{m-1} \varepsilon_m(z) z^m \sum\limits_{k=0}^{\infty} \frac{(1-\varepsilon_m(z))_{k}}{(b)_{k+m} (k+m)!} z^{k} \nonumber \\
&=(-1)^m(1+\varepsilon_m(z))_{m-1} \varepsilon_m(z) z^m \sum\limits_{k=0}^{\infty} \frac{(1)_k (1-\varepsilon_m(z))_{k}}{(b)_m (b+m)_k (1)_m (1+m)_k} \frac{z^{k}}{k!} \nonumber \\
&= \varepsilon_m(z) \tilde{q}_m(z). \nonumber
\end{align}
where 
\begin{align}
\tilde{q}_m(z):= (-1)^m(1+\varepsilon_m(z))_{m-1} z^m \sum\limits_{k=0}^{\infty} \frac{(1)_k (1-\varepsilon_m(z))_{k}}{(b)_m (b+m)_k (1)_m (1+m)_k} \frac{z^{k}}{k!}. \nonumber
\end{align}
Note, that if $z$ is large enough, then $\varepsilon_m(z)<1$, the summands in the series are positive and $\tilde{q}_m(z)$ has indeed the opposite sign of $p_m(z)$ as expected. We can rewrite $\tilde{q}_{m}(z)$ as
\begin{align}
\tilde{q}_m(z)= \frac{ (-1)^m(1+\varepsilon_m(z))_{m-1} z^m }{(1)_m (b)_m} \, {}_2F_2(1, 1 - \varepsilon_m(z); m+1, b+m;z) \nonumber
\end{align}
where 
\begin{align}
{}_2F_2(a, b; c, d;z) = \sum\limits_{k=0}^{\infty} \frac{(a)_k (b)_{k}}{(c)_k (d)_k} \frac{z^{k}}{k!}, \qquad c,d \neq 0,-1,-2,... \nonumber
\end{align}
is a generalized hypergeometric function, see for example \cite[Volume I, Chapter IV]{Bateman1953} and \cite[Chapter 3]{Luke1969}. For given parameters $a,b,c,d$, the generalized hypergeometric function ${}_2F_2$ is an entire function in $z$ and is increasing in $a,b$ for $a,b>0$. For fixed $a,b,c,d$ we have
\begin{align}
{}_2F_2(a, b; c, d;z) = \frac{\Gamma(c)\Gamma(d)}{\Gamma(a)\Gamma(b)} z^{a+b -c -d}e^z (1+O(z^{-1}))  \qquad \text{as } z\rightarrow +\infty, \label{eq:magnetic_disk_2F2_asymptotic}
\end{align}
see \cite[Chapter 5.11.3]{Luke1969} for reference. Let now $0<\varepsilon < 1$ be arbitrary and let $z^*$ be large enough such that $\varepsilon_m(z) < \varepsilon$ for all $z>z^*$. Then $(1-\varepsilon)_k < (1-\varepsilon_m(z))_k$ for all $k\in\mathbb{N}$ and all $z>z^*$, therefore 
\begin{align}
{}_2F_2(1, 1 - \varepsilon_m(z); m+1, b+m;z) > {}_2F_2(1, 1 - \varepsilon; m+1, b+m;z)>0 \nonumber
\end{align}
for all $z>z^*$ and
\begin{align}
 |\tilde{q}_m(z)| &= \frac{ (1+\varepsilon_m(z))_{m-1} z^m }{(1)_m (b)_m} \, {}_2F_2(1, 1 - \varepsilon_m(z); m+1, b+m;z) \nonumber \\
 &> \frac{ (1)_{m-1} z^m }{(1)_m (b)_m} \, {}_2F_2(1, 1 - \varepsilon; m+1, b+m;z) \nonumber \\
 &= \frac{ z^m }{m (b)_m} \frac{\Gamma(m+1)\Gamma(b+m)}{\Gamma(1-\varepsilon)} z^{1+1 - \varepsilon -m-1 -b-m}e^z (1+O(z^{-1}))  \nonumber \\
  &= \frac{\Gamma(m) \Gamma(b+m)}{ (b)_m\Gamma(1-\varepsilon)}    z^{1 - \varepsilon -b-m}e^z (1+O(z^{-1}))  \qquad \text{as } z\rightarrow +\infty. \nonumber
 \end{align}
Combining this with \eqref{eq:pm_asymptotic_simple}, we get
\begin{align}
\varepsilon_m(z) = \left| \frac{p_m(z)}{\tilde{q}_{m}(z)} \right| &\leq  \frac{ (b)_m \Gamma(1-\varepsilon) }{(b)_{m-1}\Gamma(m) \Gamma(b+m)}   z^{b+2(m-1)+ \varepsilon }e^{-z} (1+o(1))  \nonumber \\
&=  \frac{  \Gamma(1-\varepsilon)}{\Gamma(m) \Gamma(b+m-1)}    z^{b+2(m-1)+ \varepsilon }e^{-z} (1+o(1))  \qquad \text{as } z\rightarrow +\infty. \nonumber
\end{align}
and therefore for any $\varepsilon>0$
\begin{align} \label{eq:epsm_asymptotic_simple}
\varepsilon_m(z) = O( z^{b+2(m-1)+\varepsilon} e^{-z})  \qquad \text{as } z\rightarrow +\infty .
\end{align}

This is not yet the asymptotic expression for $\varepsilon_m(z)$ we want to prove, but it is a good upper bound on its decay rate. To find the asymptotic expression for $\varepsilon_m(z)$, we can feed back this upper bound in our calculation of asymptotics of $p_m(z)$ and $\tilde{q}_m(z)$ to improve them further. 

With the aid of \eqref{eq:epsm_asymptotic_simple}, we now know that for any $\varepsilon >0$ and any fixed $k\in\mathbb{N}_0$
\begin{align}
(-(m-1) - \varepsilon_m(z))_k z^k= (-(m-1))_k z^k (1+O( z^{b+2(m-1)+\varepsilon} e^{-z})) \qquad \text{as } z\rightarrow +\infty. \nonumber
\end{align}
Compared to \eqref{eq:pm_asymptotic_simple}, this leads to the improved asymptotic
\begin{align} \label{eq:pm_asymptotic_improved}
p_m(z) = \sum\limits_{k=0}^{m-1} \frac{(-(m-1))_k}{(b)_k k!} z^k + o(1) \qquad \text{as } z\rightarrow +\infty.
\end{align}
Next, we improve our lower bound for $|\tilde{q}_m(z)|$. We note that for any $k \in \mathbb{N}$ 
\begin{align}
(1-x)_{k} = \sum\limits_{j=0}^k c_{k,j} x^j \nonumber
\end{align}
where $c_{k,j}$ are suitable coefficients with $c_{k,0}=(1)_k$ and $\text{sgn}(c_{k,j}) = (-1)^j$. For $x = -1$, we see that
\begin{align}
\sum_{j=0}^k |c_{k,j}| = \sum_{j=0}^k c_{k,j} (-1)^j = (2)_k  \nonumber
\end{align}
and in particular $|c_{k,j}| \leq (2)_k $ for all $0\leq j\leq k$. Using this, we deduce that
\begin{align}
&|{}_2F_2(1, 1 ; m+1, b+m;z) - {}_2F_2(1, 1 - \varepsilon_m(z); m+1, b+m;z)| \nonumber \\
= &\; \left| \sum\limits_{k=1}^{\infty} \frac{(1)_k [(1)_k -(1-\varepsilon_m(z))_{k}]  }{ (b+m)_k (1+m)_k} \frac{z^{k}}{k!}  \right| \nonumber \\
\leq &\; \sum\limits_{k=1}^{\infty} \sum\limits_{j=1}^k \frac{(1)_k |c_{k,j}|(\varepsilon_m(z))^j   }{ (b+m)_k (1+m)_k} \frac{z^{k}}{k!} \nonumber \\
\leq &\; \sum\limits_{j=1}^{\infty} (\varepsilon_m(z))^j \sum\limits_{k=j}^\infty \frac{(1)_k (2)_{k}   }{ (b+m)_k (1+m)_k} \frac{z^{k}}{k!}. \nonumber 
\end{align}
Now, with \eqref{eq:magnetic_disk_2F2_asymptotic} follows
\begin{align}
\sum\limits_{k=j}^\infty \frac{(1)_k (2)_{k}   }{ (b+m)_k (1+m)_k} \frac{z^{k}}{k!} &\leq \sum\limits_{k=0}^\infty \frac{(1)_k (2)_{k}   }{ (b+m)_k (1+m)_k} \frac{z^{k}}{k!} \nonumber \\
&= {}_2F_2(1,2;m+1,b+m;z) \nonumber \\
&=  O( z^{-b-2(m-1)}e^z )  \qquad \text{as } z\rightarrow +\infty \nonumber
\end{align}
and with \eqref{eq:epsm_asymptotic_simple} we get
\begin{align}
\sum\limits_{j=1}^{\infty} (\varepsilon_m(z))^j \sum\limits_{k=j}^\infty \frac{(1)_k (2)_{k}   }{ (b+m)_k (1+m)_k} \frac{z^{k}}{k!}\leq &\;  \frac{\varepsilon_m(z)}{1- \varepsilon_m(z)}\, {}_2F_2(1,2;m+1,b+m;z)  \nonumber \\
=&\; O(z^{\varepsilon} ) \qquad \text{as } z\rightarrow +\infty \nonumber
\end{align}
with $\varepsilon >0$ arbitrary. Hence, we can improve the previous estimate of $|\tilde{q}_m(z)|$ to
\begin{align}
 |\tilde{q}_m(z)| &= \frac{ (1+\varepsilon_m(z))_{m-1} z^m }{(1)_m (b)_m} \, {}_2F_2(1, 1 - \varepsilon_m(z); m+1, b+m;z) \nonumber \\
 &\geq \frac{ (1)_{m-1} z^m }{(1)_m (b)_m} \, ({}_2F_2(1, 1; m+1, b+m;z) + O(z^\varepsilon)) \nonumber \\
 &= \frac{\Gamma(m) \Gamma(b+m)}{ (b)_m}   z^{1 -b-m}e^z (1+O(z^{-1}))  \qquad \text{as } z\rightarrow +\infty \nonumber
 \end{align}
which gives with \eqref{eq:pm_asymptotic_improved}
\begin{align}
\varepsilon_m(z) = \left| \frac{p_m(z)}{\tilde{q}_{m}(z)} \right| \leq    \frac{  1}{\Gamma(m) \Gamma(b+m-1)}     z^{b+2(m-1)} e^{-z} (1+O(z^{-1}))  \qquad \text{as } z\rightarrow +\infty. \label{eq:proof3_upperbnd}
\end{align}
This upper bound is the asymptotic expression we are looking for. It remains to establish the same asymptotic expression as a lower bound. But this is not difficult. Since 
\begin{align}
(1+\varepsilon_m(z))_{m-1} = (1)_{m-1} (1+O(\varepsilon_m(z))) = (1)_{m-1} (1+O(z^{-1})) \qquad \text{as } z\rightarrow +\infty,  \nonumber
\end{align}
one can estimate $\tilde{q}_m(z)$ also from above by
\begin{align}
 |\tilde{q}_m(z)| & =\frac{ (1+\varepsilon_m(z))_{m-1} z^m }{(1)_m (b)_m} \, {}_2F_2(1, 1 - \varepsilon_m(z); m+1, b+m;z) \nonumber \\
  &\leq \frac{(1)_{m-1} z^m }{(1)_m (b)_m} (1+O(z^{-1})) \, {}_2F_2(1, 1 ; m+1, b+m;z) \nonumber \\
 &= \frac{ \Gamma(m)  \Gamma(b+m)  }{ (b)_m}  z^{1-b-m}e^z (1+O(z^{-1}))  \qquad \text{as } z\rightarrow +\infty \nonumber
 \end{align}
which then yields with \eqref{eq:pm_asymptotic_improved} the corresponding lower bound
\begin{align}
\varepsilon_m(z) = \left| \frac{p_m(z)}{\tilde{q}_{m}(z)} \right| &\geq \frac{ (b)_m}{ (b)_{m-1}  \Gamma(b+m) (m-1)! }  z^{b+2(m-1)}e^{-z} (1+O(z^{-1}))  \nonumber \\
&=   \frac{  1}{\Gamma(m) \Gamma(b+m-1)}     z^{b+2(m-1)} e^{-z} (1+O(z^{-1}))  \qquad \text{as } z\rightarrow +\infty. \label{eq:proof3_lowerbnd}
\end{align}
Putting the bounds \eqref{eq:proof3_upperbnd} and \eqref{eq:proof3_lowerbnd} for $\varepsilon_m(z)$ together, we have established
\begin{align} \label{eq:eps_asymptotic}
\varepsilon_m(z) =  \frac{ 1}{\Gamma(m) \Gamma(b+m-1)}     z^{b+2(m-1)} e^{-z} (1+O(z^{-1})) \qquad \text{as } z\rightarrow +\infty.
\end{align}
Finally, when we set $b=|l|+1$, $z = BR^2/2$ and employ the relation
\begin{align}
\frac{ \lambda_{m,l}(B)}{B}  = l+ |l|+1-2a_m \left( \frac{BR^2}{2}  \right)  =  l+ |l|+1+2(m-1)+2\varepsilon_m \left( \frac{BR^2}{2}  \right) , \nonumber
\end{align}
our result \eqref{eq:eps_asymptotic} becomes the second statement of the theorem.

\section{Proof of Theorem \ref{thm:magnetic_polya_general_set} and Corollary \ref{cor:magnetic_disk_trace}} \label{sec:proof_thm2_2}

We begin with the proof of Theorem \ref{thm:magnetic_polya_general_set}. Given any $\eta>0$, we need to find an eigenvalue $\lambda_n(\Omega,B)$ that satisfies 
\begin{align} \label{eq:polya_whattoshow}
\lambda_n(\Omega,B) \leq \left( \frac{1}{2} + \eta \right) \frac{4\pi n}{|\Omega|}
\end{align}
for some $n=n(B)$ when $B$ is large enough. The index $n(B)$ cannot be chosen independent of $B$ but generally needs to increase with $B$. This is because $\lambda_n(\Omega,B)\geq B$ for any $n\in\mathbb{N}$ (for disks, this follows from Theorem \ref{thm:eigenvalue_asymptotics_disk}, but for general domains $\Omega$, see e.g.\ \cite[Lemma 1.4.1]{Fournais2010}) and therefore the bound \eqref{eq:polya_whattoshow} cannot be satisfied for  
\begin{align}
n < \left( \frac{1}{2} + \eta \right)^{-1} \frac{|\Omega|B}{4\pi }.  \nonumber
\end{align}
We will see that one can choose $n(B)$ essentially linearly increasing with $B$.

The proof of Theorem \ref{thm:magnetic_polya_general_set} is divided into three steps. First, we show the assertion of the theorem in the case where $\Omega$ is a disk. We are then able to lift the result from disks to finite unions of disjoint disks. The generalization to arbitrary domains is done using coverings by finite unions of disjoint disks and a domain inclusion argument. We conclude the section by proving the corresponding result for Riesz means.

\subsection{Proof of Theorem \ref{thm:magnetic_polya_general_set} for disks} \label{subsec:proof_of_thm_polya_disks}

Although we know the eigenvalue branches $\lambda_{m,l}(B)$ for the disk in terms of roots of Kummer's function, there remains the difficulty of relating the eigenvalue branches $\lambda_{m,l}(B)$ to the sorted eigenvalues $\lambda_n(D_R,B)$. We need to identify $\lambda_{m,l}(B)$ with $\lambda_n(D_R,B)$ for certain $n$ and also find estimates on the branches $\lambda_{m,l}(B)$. We will gain such estimates from a further investigation of the Kummer function.

It will suffice for us to characterize all eigenvalues $\lambda_n(D_R,B)\leq 3B$. Theorem \ref{thm:eigenvalue_asymptotics_disk} tells us that $\lambda_{m,l}(B) > 3B$ for $m\geq 2$ or $l>0$, so all eigenvalues $\lambda_n(D_R,B) \leq 3B$ must correspond to eigenvalue branches $\lambda_{1,l}(B)$ with $l\leq 0$. As seen in Figure \ref{fig:diamagnetism}, the number of branches $\lambda_{1,l}(B)\leq 3B$ depends on the field strength. We can count these branches by counting the intersection points between $\lambda_{1,l}(B)$ and the line $\lambda = 3B$. In general, intersections of the eigenvalue branches $\lambda_{m,l}(B)$ with lines $\lambda = (l+ |l|+1 + 2k) B$, $k\in\mathbb{N}$, can be identified with help of Lemma \ref{lem:am_zm_roots}.

Recall that $\lambda_{m,l}(B)/B$ is strictly decreasing, so $\lambda_{m,l}(B)$ has at most one intersection with a line $\lambda = (l+ |l|+1 + 2k) B$, $k\in\mathbb{N}$. By the lower bound given in Theorem \ref{thm:eigenvalue_asymptotics_disk}, the branch $\lambda_{m,l}(B)$ has no intersection with $\lambda = (l+ |l|+1 + 2k) B$ if $k\leq m-1$. If $k\geq m$, there is an intersection and Lemma \ref{lem:am_zm_roots} yields the following.

\begin{corollary} \label{cor:magnetic_disk_lambda_explicit}
Let $m\in\mathbb{N}$, $l\in\mathbb{Z}$ and $k\in\mathbb{N}$, $k\geq m$. Then $\lambda_{m,l}(B) = (l+ |l|+1 + 2k) B$ if and only if $B = 2z_m(-k,|l|+1)/R^2$. Here, $z_m(a,b)$ denotes the $m$-th positive root of $z\mapsto M(a, b, z)$.
\end{corollary}

For the proof, recall also that we denoted the $m$-th negative root of $a\mapsto M(a, b, z)$ earlier by $a_m(b,z)$, sorted in decreasing order.  

\begin{proof}
According to \eqref{eq:eigenvalue_lambda_a_m_relation}, $\lambda_{m,l}(B)$ is related to the root $a_m(b,z)$ by
\begin{align}
\lambda_{m,l}(B)= \left( l+ |l|+1 -2a_m\left(|l|+1,\frac{BR^2}{2}\right) \right) B. \nonumber
\end{align}
Therefore, $\lambda_{m,l}(B) = (l+ |l|+1 + 2k) B$ if and only if 
\begin{align}
a_m\left(|l|+1,\frac{BR^2}{2}\right) = -k. \nonumber
\end{align}
Since $z_m$ is the inverse of $a_m$ with respect to $a\in(-\infty, -(m-1))$ for any fixed $b>0$ by Lemma \ref{lem:am_zm_roots} and $-k \in (-\infty, -(m-1))$ for $k\geq m$, we see that
\begin{align}
a_m\left(|l|+1,\frac{BR^2}{2}\right) = -k \qquad \text{if and only if} \qquad \frac{BR^2}{2} = z_m(-k,|l|+1). \nonumber
\end{align}
\end{proof}

For $k\in\mathbb{N}$, the power series of $M(-k,b,z)$ turns into a polynomial of degree $k$ which allows explicit computation of the roots $z_m(-k,b)$ for low $k$. In particular, if $k=1$, the polynomial $M(-1, b, z)= 1-z/b$ is linear and its only root is $z=b$. This means $z_1(-1,b)=b$ and $a_1(b,b) = -1$. With $b=|l|+1$, the corollary implies $\lambda_{1,l}(B) = 3 B $ if and only if $l\leq 0$ and 
\begin{align}
B = \frac{2z_1(-1,|l|+1)}{R^2}= \frac{2}{R^2}(|l|+1).\nonumber
\end{align}

This allows us to give a simple characterization of sorted eigenvalues $\lambda_n(D_R,B)$.

\begin{corollary} \label{cor:magnetic_disk_lambda_lowest}
Let $n\in\mathbb{N}$. If $B \geq 2n/R^2$, then
\begin{align}
\lambda_n(D_R,B) = \lambda_{1,-(n-1)}(B). \nonumber
\end{align}
\end{corollary} 

\begin{proof}
We have just shown with Corollary \ref{cor:magnetic_disk_lambda_explicit} that 
\begin{align}
\lambda_{1,-(n-1)}(B) = 3B \qquad \text{if and only if } \qquad B=\frac{2n}{R^2}, \nonumber
\end{align}
thus, since $\lambda_{1,l}(B)/B$ is strictly monotone decreasing with respect to $B$, we have $\lambda_{1,-(n-1)}(B) \leq 3B$ for $B \geq 2n/R^2$.
Due to the fact that for any $B>0$
\begin{align}
\lambda_{1,l'}(B)> \lambda_{1,l}(B) \qquad \text{for } l' < l \leq 0, \nonumber
\end{align}
see \cite[Theorem 3.3.4]{Son2014}, we get
\begin{align}
\lambda_{1,0}(B) < \lambda_{1,-1}(B) < ... < \lambda_{1,-(n-1)}(B) \leq 3B \qquad \text{for }B \geq \frac{2n}{R^2} \nonumber
\end{align}
and $\lambda_{1,l}(B) > \lambda_{1,-(n-1)}(B)  $ for any $l \leq -n$. Recall that by Theorem \ref{thm:eigenvalue_asymptotics_disk}, if $m \geq 2$ or $l> 0$, then $\lambda_{m,l}(B) > 3B$ and hence $\lambda_{m,l}(B) > \lambda_{1,-(n-1)}(B)  $ if $B \geq 2n/R^2$. This shows that if $B \geq 2n/R^2$, there are no smaller eigenvalues than $\lambda_{1,0}(B)$, ..., $\lambda_{1,-(n-1)}(B)$ and since they are already sorted by magnitude, we can identify them with $\lambda_1(D_R,B)$, ..., $\lambda_n(D_R,B)$. 
\end{proof}

Let us now derive estimates for eigenvalue branches $\lambda_{1,l}(B)$, $l\leq 0$. First, we show the following lemma concerning Kummer's function.

\begin{lemma} \label{lem:kummer_M_delta_eps}
Let $0<\delta< 1$ and $\varepsilon>0$. There exists $b_{\delta,\varepsilon}>0$ such that for any $b>b_{\delta,\varepsilon}$
\begin{align}
M(-\delta, b, (1+\varepsilon) b) < 0. \nonumber
\end{align}
\end{lemma}

\begin{proof}
First observe that
\begin{align}
M(-\delta, b, (1+\varepsilon) b)&= 1+\sum\limits_{k=1}^\infty \frac{(-\delta)_k}{k!} \frac{b^k}{(b)_k}  (1+\varepsilon)^k \nonumber
\end{align}
where the series only sums negative terms when $0<\delta< 1$. The appearing series converges, however the series
\begin{align}
\sum\limits_{k=1}^\infty \frac{(-\delta)_k}{ k!}  (1+\varepsilon)^k, \nonumber
\end{align}
where we replaced the $b$-dependent fraction by one, is divergent for any $\varepsilon>0$. It tends to $-\infty$, since all summands are negative and the radius of convergence of the power series 
\begin{align}
\sum\limits_{k=1}^\infty \frac{(-\delta)_k}{ k!}  z^k \nonumber
\end{align}
is
\begin{align}
R = \left( \lim_{k \rightarrow +\infty } \left(\frac{(-\delta)_k}{ k!} \right)^\frac{1}{k} \right)^{-1}=\left( \lim_{k \rightarrow +\infty } \left(\frac{\Gamma(-\delta+k)}{\Gamma(-\delta) \Gamma(k+1)} \right)^\frac{1}{k} \right)^{-1} = 1. \nonumber
\end{align}
We can therefore choose some $c \in (0,1)$ and then some large $N\in\mathbb{N}$ such that  
\begin{align}
 c \sum\limits_{k=1}^N \frac{(-\delta)_k}{ k!}  (1+\varepsilon)^k <- 1. \nonumber
\end{align}
Furthermore, for any fixed $k\geq 1$ holds
\begin{align}
\frac{b^k}{(b)_k } \rightarrow 1 \qquad \text{as } b \rightarrow +\infty, \nonumber
\end{align}
thus there exists $b_{\delta,\varepsilon}>0$ such that for $b > b_{\delta,\varepsilon}$ 
\begin{align}
\frac{b^k}{(b)_k }\geq c \qquad \text{for all }  k=1,...,N. \nonumber
\end{align}
Hence, for $b > b_{\delta,\varepsilon}$
\begin{align}
M(-\delta, b, (1+\varepsilon) b)= 1+\sum\limits_{k=1}^\infty \frac{(-\delta)_k}{k!} \frac{b^k}{(b)_k}  (1+\varepsilon)^k \leq 1 + c \sum\limits_{k=1}^N \frac{(-\delta)_k}{ k!}  (1+\varepsilon)^k < 0. \nonumber
\end{align}
and we have shown the lemma.
\end{proof}

The lemma implies a bound on the roots in $a$ of Kummer's function and therefore on eigenvalue branches. With the assertion of the lemma and $M(0,b,(1+\varepsilon)b)=1>0$, we see that for any $b > b_{\delta,\varepsilon}$, the map $a \mapsto M(a, b, (1+\varepsilon) b )$ has at least one root in $ (-\delta,0)$ by the intermediate value theorem. In particular, we have 
\begin{align} \label{eq:a_1_estimate_delta}
-\delta < a_1(b,(1+\varepsilon)b) <0 \qquad \text{for }b > b_{\delta,\varepsilon}.
\end{align}
Again, recall that $a_1(b,(1+\varepsilon)b)$ relates to the eigenvalues $\lambda_{1,l}(B)$ via \eqref{eq:eigenvalue_lambda_a_m_relation}. When setting $b=|l|+1$ for $l\leq 0$ we see that
\begin{align}
\frac{ \lambda_{1,l}(B_{|l|+1})}{B_{|l|+1}} = 1 - 2a_1(|l|+1,(1+\varepsilon)(|l|+1)) \nonumber
\end{align}
where the magnetic field strength $B_{|l|+1}$ is given by
\begin{align}
B_{|l|+1} = \frac{2}{R^2}(1+ \varepsilon) (|l|+1). \nonumber
\end{align}
But then \eqref{eq:a_1_estimate_delta} implies that 
\begin{align}
\frac{ \lambda_{1,l}(B_{|l|+1})}{B_{|l|+1}} \leq 1+2\delta \qquad \text{for any } l\leq 0, |l| +1 > b_{\delta,\varepsilon} \nonumber
\end{align}
and since $\lambda_{1,l}(B)/B$ is strictly decreasing with $B$, we have 
\begin{align} \label{eq:lambda_1l_Bl1_estimate}
\frac{ \lambda_{1,l}(B)}{B} \leq 1+2\delta \qquad \text{for any } B\geq B_{|l|+1} \text{ and } l\leq 0, |l| +1 > b_{\delta,\varepsilon}.
\end{align}

With \eqref{eq:lambda_1l_Bl1_estimate} at hand, we are almost done. It remains to bring $B$ to the right hand side, choose $\delta$ and $\varepsilon$ appropriately and use Corollary \ref{cor:magnetic_disk_lambda_lowest} to identify $\lambda_{1,-(n-1)}(B)$ with $\lambda_n(D_R, B)$.

Let $\eta > 0$ be arbitrary. Choose $\varepsilon>0$ and $0<\delta<1$ small enough such that $(1+2\delta)(1+ \varepsilon) < 1+ 2 \eta$. After that choose $b^* \geq b_{\delta,\varepsilon}$ large enough such that 
\begin{align}
(1+2\delta)(1+ \varepsilon) \left( 1+ \frac{1}{b^*} \right)\leq 1+ 2 \eta. \nonumber
\end{align}
Assume $|l|+1 > b^* $ and 
\begin{align}
B = \frac{2}{ R^2}(1+ \tilde{\varepsilon}) (|l|+1) , \qquad \text{where } \tilde{\varepsilon} \in \left[ \varepsilon, \varepsilon + \frac{1+\varepsilon}{b^*} \right].  \label{eq:magnetic_disk_Bl_interval} 
\end{align}
Since $\tilde{\varepsilon}\geq \varepsilon$, we have $B\geq B_{|l|+1}$ and therefore \eqref{eq:lambda_1l_Bl1_estimate} holds. Hence, it follows that
\begin{align}
\lambda_{1,l}(B)&\leq (1+2\delta) B = (1+2\delta)\frac{2}{ R^2}(1+ \tilde{\varepsilon}) (|l|+1) \nonumber \\
&\leq (1+2\delta)\left(1+ \varepsilon + \frac{1+\varepsilon}{b^*}\right) \frac{2\pi}{|D_R|}(|l|+1) \nonumber \\
&=(1+2\delta)(1+ \varepsilon) \left( 1 + \frac{1}{b^*}\right)\frac{2\pi}{|D_R|}(|l|+1) \nonumber \\
&\leq (1+ 2 \eta) \frac{2\pi }{|D_R|} (|l|+1). \nonumber
\end{align}
Note that if $|l|+1 > b^* $ and $B \in [B_{|l|+1}, B_{|l|+2}]$ then $B$ can be written as in \eqref{eq:magnetic_disk_Bl_interval} and thus we have shown that
\begin{align}
\lambda_{1,l}(B) \leq (1+ 2 \eta) \frac{2\pi }{|D_R|} (|l|+1), \qquad \text{for } B \in [B_{|l|+1}, B_{|l|+2}], \; l\leq 0, \; |l|+1 > b^*. \nonumber
\end{align}
Now let $n = -l + 1 = |l|+1$. Then $B \in [B_{|l|+1}, B_{|l|+2}]$ implies $B > 2n/R^2$ and therefore by Corollary \ref{cor:magnetic_disk_lambda_lowest}
\begin{align} \label{eq:polya_disk_proof_eventually}
\lambda_n(D_R, B) = \lambda_{1,-(n-1)}(B) = \lambda_{1,l}(B) \leq (1+ 2 \eta) \frac{2\pi }{|D_R|} (|l|+1) =  \left(\frac{1}{2} + \eta \right) \frac{4\pi n}{|D_R|}  
\end{align}
for $B \in [B_{n}, B_{n+1}], \; n> b^*$. The desired inequality hence always holds for some eigenvalue if $B \geq \min_{n > b^*} B_{n} $.

\begin{remark}
We remind the reader that Theorem \ref{thm:magnetic_polya_general_set} is equivalent to \eqref{eq:lim_polya_excess_factor}. The convergence rate of the limit in \eqref{eq:lim_polya_excess_factor} can be quantified in case of the disk. The key idea is to improve Lemma \ref{lem:kummer_M_delta_eps} by letting $\delta$ and $\varepsilon$ depend on $b$ and such that
\begin{align}
M(-\delta(b),b,b(1+\varepsilon(b)))<0 \nonumber
\end{align}
still holds for $b$ large enough. This eventually yields an upper bound of the form \eqref{eq:polya_disk_proof_eventually} with $\eta=\eta(b)$ depending on the choice of functions $\delta=\delta(b)$ and $\varepsilon=\varepsilon(b)$. By optimizing $\delta(b)$, $\varepsilon(b)$ and relating $b$ to $B$, one can derive explicit convergence rates of the form \eqref{eq:polya_excess_factor_landau}.
\end{remark}

\subsection{Proof of Theorem \ref{thm:magnetic_polya_general_set} for disjoint unions of finitely many disks} \label{subsec:proof_of_thm_disjoint_disks}

Let $\eta>0$ be arbitrary. Let $N\in\mathbb{N}$ and $\Omega = \Omega_1 \cup ... \cup \Omega_N$ be a disjoint union of $N$ disks $\Omega_i=D_{R_i}$ with radii $R_i>0$. Then of course $|\Omega| = |\Omega_1| + ... + |\Omega_N|=\pi \sum_{i=1}^N R_i^2$. By the previous proof for a single disk, we can choose a small $\varepsilon>0$ so that for all $i$ exist $b^*_i$ such that
\begin{align}
\lambda_n(\Omega_i, B) \leq (1+ 2 \eta) \frac{2\pi n}{|\Omega_i|} , \qquad \text{for } B \in [B_{n}^{(i)}, B_{n+1}^{(i)}],  \; n > b^*_i, \nonumber
\end{align}
where 
\begin{align}
B_{n}^{(i)} = \frac{2}{R_i^2}(1+\varepsilon)n. \nonumber
\end{align}
Let $\eta'>0$ and let $B_{\eta'}>0$ be large enough such that
\begin{align}
\frac{1}{1-\frac{2\pi (1+\varepsilon)N}{|\Omega|B}}\leq 1+\eta' \nonumber
\end{align} 
for $B\geq B_{\eta'}$. Then assume $B\geq  \max\lbrace \max_{1\leq i \leq N} B_{{\lceil b_i^* \rceil }}^{(i)}, B_{\eta'} \rbrace$. For such $B$, there exists for all $i$ an $n_i$ with $B \in [B_{n_i}^{(i)}, B_{n_i+1}^{(i)}]$ and we have seen that for this $n_i$ holds
\begin{align}
\lambda_{n_i}(\Omega_i, B) \leq (1+ 2 \eta) \frac{2\pi n_i}{|\Omega_i|} . \nonumber
\end{align}
The spectrum of $\Omega$ is just the union of the spectra of all $\Omega_i$, so for any $i$ exists an $n_i'$ such that $\lambda_{n_i'}(\Omega,B) = \lambda_{n_i}(\Omega_i,B)$. We now let $n:= \max_{1\leq i \leq N} n_i' $ so that $\lambda_n(\Omega,B)$ is the eigenvalue of $\Omega$ whose index corresponds to the largest of the eigenvalues $\lambda_{n_i}(\Omega_i,B)$. The index of the maximizing component shall be denoted by $i^*$ so that we have $\lambda_n(\Omega,B)= \lambda_{n_{i^*}}(\Omega_{i^*},B)$. By the choice of $n$ we necessarily have $n \geq  \sum_{i=1}^N n_i$. With this, we observe that
\begin{align}
\lambda_n(\Omega, B) =\lambda_{n_{i^*}}(\Omega_{i^*}, B) &\leq (1+ 2 \eta) \frac{2\pi n_{i^*}}{|\Omega_{i^*}|}  =  \frac{n_{i^*}}{n} \frac{|\Omega|}{|\Omega_{i^*}|} \left(\frac{1}{2}+\eta \right)  \frac{4\pi n}{|\Omega|}    , \nonumber
\end{align}
so all that remains to show is that $\frac{n_{i^*}}{n} \frac{|\Omega|}{|\Omega_{i^*}|}$ is close to one. From $B \in [B_{n_i}^{(i)}, B_{n_i+1}^{(i)}]$ follows
\begin{align} \label{eq:disjoint_Ri_estimate1}
\frac{B R_i^2}{2(1+\varepsilon)} -1 \leq n_i \leq \frac{B R_i^2}{2(1+\varepsilon)} 
\end{align}
for any $i$. Summing these inequalities over $i$ yields
\begin{align} \label{eq:disjoint_Ri_estimate2}
\frac{B \sum_{i=1}^N R_i^2}{2(1+\varepsilon)} -N \leq \sum_{i=1}^N n_i \leq \frac{B \sum_{i=1}^N R_i^2}{2(1+\varepsilon)} .
\end{align}
Since $n\geq \sum_{i=1}^N n_i$, the inequalities \eqref{eq:disjoint_Ri_estimate1} and \eqref{eq:disjoint_Ri_estimate2} imply
\begin{align}
\frac{n_{i^*}}{n} \leq \frac{\frac{B R_{i^*}^2}{2(1+\varepsilon)} }{\frac{B \sum_{i=1}^N R_i^2}{2(1+\varepsilon)} -N} =  \frac{ |\Omega_{i^*}|}{ |\Omega| -\frac{2\pi (1+\varepsilon)N}{B}} \nonumber
\end{align}
and thus
\begin{align}
\frac{n_{i^*}}{n} \frac{|\Omega|}{|\Omega_{i^*}|} \leq  \frac{1}{1-\frac{2\pi (1+\varepsilon)N}{|\Omega|B}}\leq 1+\eta'. \nonumber
\end{align}
Hence, we have 
\begin{align}
\lambda_n(\Omega, B) \leq (1 + \eta')   \left(\frac{1}{2}+\eta \right)  \frac{4\pi n}{|\Omega|}  . \nonumber
\end{align}
Because $\eta>0$ and $\eta'>0$ were arbitrary, this proves the assertion if $\Omega$ is a disjoint union of finitely many disks.

\subsection{Proof of Theorem \ref{thm:magnetic_polya_general_set} for general domains} \label{subsec:proof_of_thm_general_domains}

For the final step, we take advantage of the domain inclusion principle. This principle is well-known for the non-magnetic Dirichlet Laplacian and states that eigenvalues increase under domain inclusion. The same holds with a constant magnetic field.

\begin{lemma}[Domain inclusion principle] \label{prop:basics_domain_inclusion_principle}
Let $\Omega, \Omega'$ be open domains of finite measure with $\Omega' \subset \Omega$ and $B\geq 0$. Then 
\begin{align}
\lambda_n(\Omega',B) \geq \lambda_n(\Omega,B) \qquad \text{for } n \in\mathbb{N}. \nonumber
\end{align}
\end{lemma}

\begin{proof}
Since $C_0^\infty(\Omega') \subset C_0^\infty(\Omega)$ by trivial extension, the form closure of $C_0^\infty(\Omega')$ w.r.t.~the quadratic form of the magnetic Laplacian can be embedded in the form closure of $C_0^\infty(\Omega)$. The min-max-principle then shows the inequality between the eigenvalues.
\end{proof}

Let now $\Omega$ be any open set of finite measure. By Vitali's covering theorem there exists a countable collection of pairwise disjoint open disks $\lbrace \Omega_i \rbrace_{i=1}^\infty$ such that $\Omega_i\subset \Omega$ for all $i\in\mathbb{N}$ and
\begin{align}
\left| \Omega \setminus \bigcup_{i=1}^\infty \Omega_i\right| = 0. \nonumber
\end{align}
This means in particular that
\begin{align}
\left| \bigcup_{i=1}^N \Omega_i \right| \rightarrow |\Omega|  \qquad \text{as } N\rightarrow +\infty. \nonumber
\end{align}
In other words, for any $\eta'>0$ exists $\Omega'\subset \Omega $, $\Omega'$ being a disjoint union of finitely many disks, with $\dfrac{|\Omega|}{|\Omega'|} \leq 1+\eta'$. We have for $B$ large enough that 
\begin{align}
\lambda_n(\Omega', B) \leq (1+2\eta)  \frac{2\pi n}{|\Omega'|} . \nonumber
\end{align}
for some $n$. But then by domain monotonicity we arrive with such a disjoint union of disks $\Omega'$ at
\begin{align}
\lambda_n(\Omega, B) \leq \lambda_n(\Omega', B) \leq (1+2\eta)   \frac{2\pi n}{|\Omega'|}  \leq (1+2\eta)   \dfrac{|\Omega|}{|\Omega'|} \frac{2\pi n}{|\Omega|}  =(1 + \eta')  \left(\frac{1}{2}+\eta \right)   \frac{4\pi n}{|\Omega|} . \nonumber
\end{align}
Again, since $\eta$ and $\eta'>0$ were arbitrary, this shows the assertion for general $\Omega$ and finishes the proof.

\subsection{Proof of Corollary \ref{cor:magnetic_disk_trace}}

Corollary \ref{cor:magnetic_disk_trace} follows from Theorem \ref{thm:magnetic_polya_general_set}. We begin with the case $\gamma=0$. 

Let $\eta>0$ be arbitrary. Choose $\eta'>0$ small enough such that $2-\eta < 2 ( 1+2\eta')^{-1}$. By Theorem \ref{thm:magnetic_polya_general_set}, there exists $B_{\eta'}>0$ such that for any $B\geq B_{\eta'}$ exists an $n$ with
\begin{align}
2-\eta  < 2 ( 1+2\eta')^{-1}  \leq  \frac{4\pi n}{|\Omega| \lambda_n(\Omega, B)}   \leq 2 . \nonumber
\end{align}
Recall that $R_0 = 2$ and $L_{0,2}^{\mathrm{cl}} = 1/4\pi$. Thus, letting $\lambda \searrow \lambda_n(\Omega, B)$, we get for any $B\geq B_{\eta'}$
\begin{align} \label{eq:riesz_mean_delta_gamma0}
R_0 -\eta \leq \frac{\mathrm{tr} (H_B^\Omega - \lambda)_-^0 }{  L_{0,2}^{\mathrm{cl}}  |\Omega| \lambda }  \leq R_0. 
\end{align}
This is \eqref{eq:cor_magnetic_disk_trace} for $\gamma = 0$.

Let now $0<\gamma < 1$ and let $B\geq 0$ be fixed. Suppose there exists $R< R_\gamma$ with
\begin{align} \label{eq:cor_proof_1}
\mathrm{tr} (H_B^\Omega - \mu)_-^\gamma \leq R L_{\gamma,2}^{\mathrm{cl}}  |\Omega| \mu^{1+\gamma}
\end{align}
for all $\mu>0$. By \cite[Lemma 3.31]{Frank2022}, it is known that for any $\mu > \lambda$
\begin{align}\label{eq:cor_proof_2}
\mathrm{tr} (H_B^\Omega - \lambda)_-^0 \leq (\mu - \lambda)^{-\gamma} \mathrm{tr} (H_B^\Omega - \mu)_-^\gamma
\end{align}
and for any $0\leq \alpha < \beta$
\begin{align}\label{eq:cor_proof_3}
\inf_{\mu > \lambda} (\mu - \lambda)^{-\alpha} \mu^\beta =C(\alpha,\beta)^{-1} \lambda^{\beta-\alpha}.
\end{align}
where $C(\alpha,\beta):=\alpha^\alpha\beta^{-\beta}  (\beta-\alpha)^{\beta-\alpha}$. But from \eqref{eq:cor_proof_2} follows under assumption of \eqref{eq:cor_proof_1}
\begin{align}
\mathrm{tr} (H_B^\Omega - \lambda)_-^0 &\leq R L_{\gamma,2}^{\mathrm{cl}}  |\Omega| (\mu - \lambda)^{-\gamma}  \mu^{1+\gamma}  \nonumber
\end{align}
for any $\mu>\lambda>0$ and hence with \eqref{eq:cor_proof_3}
\begin{align}
\mathrm{tr} (H_B^\Omega - \lambda)_-^0\leq R C(\gamma,1+\gamma)^{-1} L_{\gamma,2}^{\mathrm{cl}}  |\Omega|  \lambda. \nonumber
\end{align}
This however yields 
\begin{align}
\frac{\mathrm{tr} (H_B^\Omega - \lambda)_-^0 }{  L_{0,2}^{\mathrm{cl}}  |\Omega| \lambda } \leq  R C(\gamma,1+\gamma)^{-1} \frac{L_{\gamma,2} ^{\mathrm{cl}} }{L_{0,2}}  <  R_\gamma \gamma^{-\gamma}(1+\gamma)^{\gamma}  = R_0 \nonumber
\end{align}
which contradicts the lower bound of \eqref{eq:riesz_mean_delta_gamma0}, if $\eta$ is small enough and if $B\geq B_{\eta'}$. Therefore, \eqref{eq:cor_magnetic_disk_trace} is also true for $0<\gamma < 1$.

For $\gamma\geq 1$, we have already discussed that \eqref{eq:cor_magnetic_disk_trace} is a consequence of the asymptotics for Riesz means.

\backmatter


\bmhead{Acknowledgements}

The first author is grateful for the hospitality by the Institut Mittag-Leffler and support by COST (Action CA18232) in June 2023. He also wants to express his gratitude to the organizers of the International Conference on Spectral Theory and Approximation 2023 in Lund and the granted financial support. Finally, he wants to thank S{\o}ren Fournais, Mikael Persson Sundqvist, Dirk Hundertmark and Semjon Vugalter for valuable discussions.

\section*{Declarations}

\begin{itemize}
\item Funding: No funding was received to assist with the preparation of this manuscript.
\item Author contribution: The first draft and subsequent revisions of the manuscript were written by Matthias Baur. Timo Weidl read and suggested improvements on previous versions of the manuscript. All authors read and approved the final manuscript.
\item Competing interests: The authors have no competing interests to declare.
\item Ethics approval and consent to participate: Not applicable
\item Consent for publication: Not applicable
\item Data availability: Not applicable 
\item Materials availability: Not applicable
\item Code availability: Not applicable
\end{itemize}

\bibliography{literature.bib}

\end{document}